\documentclass[10pt,a4paper]{article}

\usepackage{a4wide}
\usepackage{setspace}
\usepackage{amsmath,amssymb,amsfonts,amsthm}
\usepackage{mathtools}
\usepackage[table]{xcolor}
\usepackage[numbers,sort&compress]{natbib}
\usepackage{authblk}
\usepackage{comment}
\usepackage{subcaption}
\usepackage{framed}

\usepackage{mdframed}

\numberwithin{equation}{section}
\mathtoolsset{showonlyrefs}

\newtheorem{theorem}{Theorem}[section]
\newtheorem{hypothesis}[theorem]{Hypothesis}

\newmdtheoremenv[nobreak]{proposition}[theorem]{Proposition} 
\newmdtheoremenv[nobreak]{corollary}[theorem]{Corollary}
\newmdtheoremenv[backgroundcolor=black!5,nobreak]{problem}[theorem]{Problem}
\newmdtheoremenv[backgroundcolor=black!5,nobreak]{definition}[theorem]{Definition}

\usepackage[table]{xcolor}
\usepackage{color}
\usepackage[colorlinks=true,citecolor=red,linkcolor=magenta]{hyperref}

\usepackage{tikz}
\usetikzlibrary{calc}
\usetikzlibrary{decorations.markings}

\newcommand{\im}{\imath}

\renewcommand{\vec}[1]{{\boldsymbol{\mathbf{#1}}}}
\newcommand{\tens}[1]{{\boldsymbol{\mathbf{\mathit{#1}}}}}

\newcommand{\ALPHA}{\tens{\alpha}}
\newcommand{\EPS}{\tens{\varepsilon}}

\newcommand{\ZETA}{\tens{\zeta}}
\newcommand{\MU}{\tens{\mu}}
\newcommand{\XI}{\tens{\xi}}
\newcommand{\RHO}{\tens{\rho}}

\newcommand{\D}{{\mathrm{D}}}
\newcommand{\N}{{\mathrm{N}}}
\newcommand{\R}{{\mathrm{R}}}

\DeclareMathOperator{\curl}{\mathbf{curl}}
\DeclareMathOperator{\Curl}{\mathbf{curl}}
\DeclareMathOperator{\scurl}{\mathrm{curl}}

\DeclareMathOperator{\sdiv}{\mathrm{div}}

\newcommand{\ba}{\mathbf{a}}
\newcommand{\bb}{\mathbf{b}}

\newcommand{\be}{\mathbf{e}}
\renewcommand{\bf}{\mathbf{f}}
\newcommand{\bg}{\mathbf{g}}
\newcommand{\bh}{\mathbf{h}}

\newcommand{\bj}{\mathbf{j}}

\newcommand{\bn}{\mathbf{n}}

\newcommand{\br}{\mathbf{r}}

\newcommand{\bu}{\mathbf{u}}
\newcommand{\bv}{\mathbf{v}}

\newcommand{\bx}{\mathbf{x}}
\newcommand{\by}{\mathbf{y}}
\newcommand{\bz}{\mathbf{z}}
\newcommand{\bzero}{\mathbf{0}}

\newcommand{\BH}{\mathbf{H}}

\newcommand{\BL}{\mathbf{L}}

\newcommand{\BV}{\mathbf{V}}

\title{Analysis of time-harmonic electromagnetic problems \\ with elliptic material coefficients\footnote{Submitted for publication. Distributed under \href{https://creativecommons.org/licenses/by/4.0/}{Creative Commons CC BY 4.0} license.}}

\author[1]{Patrick Ciarlet Jr.}
\author[1]{Axel Modave}

\affil[1]{\footnotesize POEMS, CNRS, Inria, ENSTA Paris, Institut Polytechnique de Paris, 91120 Palaiseau, France. Corresponding authors: \href{mailto:patrick.ciarlet@ensta-paris.fr}{\texttt{patrick.ciarlet@ensta-paris.fr}}, \href{mailto:axel.modave@ensta-paris.fr}{\texttt{axel.modave@ensta-paris.fr}}}

\date{}

\begin{document}

\maketitle

\begin{abstract}
    \noindent
    We consider time-harmonic electromagnetic problems with material coefficients represented by elliptic fields, covering a wide range of complex and anisotropic material media.
    The properties of elliptic fields are analyzed, with particular emphasis on scalar fields and normal tensor fields.
    Time-harmonic electromagnetic problems with general elliptic material fields are then studied.
    Well-posedness results for classical variational formulations with different boundary conditions are reviewed, and hypotheses for the coercivity of the corresponding sesquilinear forms are investigated.
    Finally, the proposed framework is applied to examples of media used in the literature: isotropic lossy media, geometric media, and gyrotropic media.
\end{abstract}


\setlength{\parskip}{2mm}
\setstretch{1.1}

\section{Introduction}

The mathematical modeling of electromagnetic wave propagation in complex and anisotropic media is a very active research topic, e.g.~for the design of metamaterials.
In this article, we consider electromagnetic models defined with linear boundary value problems in the time-harmonic regime, which are often considered in practice.
Knowledge of the mathematical properties of these problems is of paramount importance for the selection and development of accurate, reliable and efficient numerical solution methods, e.g.~based on finite element methods and domain decomposition methods.

The mathematical analysis of time-harmonic electromagnetic problems has been carried out for isotropic media and certain classes of anisotropic media.
Variational formulations with symmetric positive definite material tensor fields have been studied for example in \cite{hiptmair2002finite,monk2003finite,kirsch2015mathematical,assous2018mathematical}.
Non-Hermitian material tensors have been considered in \cite{alonso1998unique} for complex symmetric tensors, and in \cite{back2015electromagnetic} for anisotropic media coming from plasma theory.
Well-posedness and regularity results have been recently obtained in \cite{chicaud2021analysis,chicaud2021phd,chicaud2023analysis} for problems with material tensor fields verifying a general ellipticity condition.

In this article, we investigate the properties of elliptic fields as defined in \cite{chicaud2021analysis,chicaud2021phd,chicaud2023analysis}, and we review and systematically analyze time-harmonic electromagnetic problems with elliptic material coefficients.
We particularly focus on settings with elliptic scalar fields (e.g.~for isotropic lossy media) and elliptic normal and Hermitian tensor fields (e.g.~for geometric and gyrotropic media).

This article is structured as follows.
In Section \ref{sec:model}, the mathematical model is described.
The elliptic material fields are defined and studied in Section \ref{sec:fields}.
In Section \ref{sec:problems}, the variational formulations of time-harmonic electromagnetic problems with Dirichlet, Neumann and Robin boundary conditions are defined with the appropriate functional spaces, well-posedness properties are reviewed, and the coercivity of the sesquilinear forms is studied.
Finally, in Section \ref{sec:applications}, the framework is applied to examples of media used in the literature: isotropic lossy media, geometric media, and gyrotropic media.

\section{Mathematical model}
\label{sec:model}

We consider time-harmonic electromagnetic problems posed in a domain $\Omega\subset\mathbb{R}^3$, which is assumed to be an open, connected, bounded region with a Lipschitz-continuous boundary $\Gamma := \partial\Omega$.
The problems are written with the time-harmonic Maxwell equations,
\begin{align}
    -\im\omega\MU \bh + \curl\be & = \bzero, \label{eqn:maxwell1} \\
    -\im\omega\EPS\be - \curl\bh & = -\bj,   \label{eqn:maxwell2}
\end{align}
where $\be$ and $\bh$ are the electric and magnetic fields, $\EPS$ and $\MU$ are the electric permittivity tensor and the magnetic permeability tensor, $\bj$ is the current density, and $\omega$ is the angular frequency.
We assume that $\omega$ is a strictly positive real constant, and we take the convention that the time-dependence of the fields is $e^{-\im\omega t}$, where $t$ is the time.
Boundary conditions are prescribed on the boundary $\Gamma$.
Standard boundary conditions are
\begin{align}
    \be\times\bn                                                 & = \bg_\be,                              \\
    \bh\times\bn                                                 & = \bg_\bh,                              \\
    \bn\times(\bh\times\bn) + \ALPHA_\mathrm{imp} (\be\times\bn) & = \bg_\mathrm{imp}, \label{eqn:bnd:imp}
\end{align}
where $\bn$ is the unit outward normal, $\ALPHA_\mathrm{imp}$ is an impedance tensor and $\bg_\be$, $\bg_\bh$ and $\bg_\mathrm{imp}$ are surface data.
The Silver-M\"uller boundary condition is a particular case of the third boundary condition.
If the material tensors are real isotropic with $\EPS=\varepsilon\tens{I}_3$ and $\MU=\mu\tens{I}_3$, the Silver-M\"uller boundary condition corresponds to the surface tensor $\ALPHA_\mathrm{imp}=\sqrt{\varepsilon/\mu}\:\tens{I}_2$, where $\tens{I}_3$ is the 3-by-3 identity tensor and $\tens{I}_2$ is the 2-by-2 identity tensor defined along the tangent plane of the boundary $\Gamma$, that is $\tens{I}_2 = \tens{I}_3 - \vec{n}\otimes\vec{n}$.
In an abuse of notation, in equation~\eqref{eqn:bnd:imp}, we assume that $\ALPHA_\mathrm{imp}$ acts only on the tangent components of $\be\times\bn$.

The $\be$-formulations of these problems are obtained by removing $\bh$ from equation~\eqref{eqn:maxwell2} and from the boundary conditions thanks to equation~\eqref{eqn:maxwell1}, which leads to the equation
\begin{align}
    \curl(\MU^{-1}\curl\be) - \omega^2\EPS\be & = \bf
    \label{eqn:MaxwellEqn}
\end{align}
and the boundary conditions
\begin{align}
    \be\times\bn                                                   & = \bg_\D,
    \label{eqn:DirichetBC}                                                     \\
    (\MU^{-1}\curl\be)\times\bn                                    & = \bg_\N,
    \label{eqn:NeumannBC}                                                      \\
    \bn\times((\MU^{-1}\curl\be)\times\bn) + \ALPHA (\be\times\bn) & = \bg_\R,
    \label{eqn:RobinBC}
\end{align}
which correspond, respectively, to Dirichlet, Neumann and Robin boundary conditions.
Let us note that $\bf=\im\omega\bj$, $\bg_\D=\bg_\be$, $\bg_\N=\im\omega\bg_\bh$ and $\bg_\R=\im\omega\bg_\mathrm{imp}$ and $\ALPHA=\im\omega\ALPHA_\mathrm{imp}$.

In this article, we consider that $\MU(\bx)$ and $\EPS(\bx)$ are general complex 3-by-3 tensor fields, and that $\ALPHA(\bx)$ is a general complex 2-by-2 tensor field.
We assume that these tensor fields verify the ellipticity condition defined in the next section.
\section{Elliptic material fields}
\label{sec:fields}

In this section, we define and study a general family of scalar and tensor material fields.
The set $D$ is a bounded subset of $\mathbb{R}^3$, $\XI(\bx)$ is a complex $d$-by-$d$ tensor field defined on $D$ (for $d=1$, $2$ or $3$), and $\tens{L}^\infty(D) := \{\XI$ Lebesgue-measurable such that $\|\XI\|_{\tens{L}^\infty(D)} < +\infty\}$ with $\|\XI\|_{\tens{L}^\infty(D)} := \mathrm{esssup}_{\bx\in D}\sup_{\bz\in\mathbb{C}^d\backslash\{0\}}|\XI\bz|/|\bz|$.
In the scalar case, i.e.~for $d=1$, we use the notations $\xi(\bx)$, $L^\infty(D)$ and $\|\xi\|_{L^\infty(D)}$.


\subsection{Definition of elliptic fields and ellipticity directions}

We consider a general family of tensor fields satisfying the following \emph{ellipticity condition}.
\begin{definition}[Elliptic tensor field]
  \label{def:ellipticityTens}
  The $d$-by-$d$ tensor field $\XI\in \tens{L}^\infty(D)$ is said to be \emph{elliptic} if
  \begin{align}
    \exists(\theta_\XI,\xi_-)\in\mathbb{R}\times\mathbb{R}_{>0}, \ \
    \text{a.e.~in $D$}, \ \
    \forall\bz\in\mathbb{C}^d, \ \
    \Re\{e^{\im\theta_\XI}\:(\XI\bz)\cdot\overline{\bz}\}\geq\xi_-|\bz|^2.
    \label{eqn:ellipticityTens}
  \end{align}
  In addition, we use the notation $\xi_+ := \|\XI\|_{\tens{L}^\infty(D)}$.
\end{definition}
\noindent
The ellipticity condition \eqref{eqn:ellipticityTens} means that there exists at least one direction of the complex plane, called an \emph{ellipticity direction}, along which the tensor is coercive a.e.
\begin{definition}[Ellipticity direction]
  An \emph{ellipticity direction of $\XI$} is an angle $\theta_\XI$ for which there exists $\xi_-\in\mathbb{R}_{>0}$ such that the pair $(\theta_\XI,\xi_-)$ fulfills the ellipticity condition~\eqref{eqn:ellipticityTens}.
\end{definition}
\noindent
If $\theta_\XI$ is an ellipticity direction, then $\theta_\XI+2m\pi$ is also an ellipticity direction for every integer $m$.
In practice, the sets of ellipticity directions can be written in a certain range to simplify comparisons, although they must be understood up to a multiple of $2\pi$.
In this work, we take the convention to write them in the range $[-\pi,\pi]$.
\begin{definition}[$\Theta$-set]
  The \emph{$\Theta$-set of $\XI$}, denoted by $\Theta_\XI$, is the set of all the ellipticity directions of $\XI$ belonging to the range $[-\pi,\pi]$.
\end{definition}
For a given ellipticity direction $\theta\in\Theta_\XI$, there exists a range of values ${]0,\xi_-^\mathrm{sup}[}$, such that, for all $\xi_-\in{]0,\xi_-^\mathrm{sup}[}$, the pair $(\theta,\xi_-)$ fulfills the ellipticity condition \eqref{eqn:ellipticityTens}.
In this sense, we say that $\xi_-^\mathrm{sup}$, and then $\xi_-$, depends on $\theta_\XI$.

\noindent
The following definition is introduced for configurations where two elliptic tensors share at least one ellipticity direction.
\begin{definition}[Simultaneous ellipticity]
  Two elliptic tensors $\XI_1$ and $\XI_2$ are said to be \emph{simultaneously elliptic} if $\Theta_{\XI_1}\cap\Theta_{\XI_2}\neq\emptyset$.
\end{definition}


\subsection[Properties of elliptic scalar fields]{Properties of elliptic scalar fields}

For scalar fields, the ellipticity condition can be simplified, and Definition \ref{def:ellipticityTens} can be rewritten as follows.
\begin{definition}[Elliptic scalar field]
  \label{def:ellipticitySca}
  The scalar field $\xi(\bx) = |\xi(\bx)|\:e^{\im\varphi_\xi(\bx)} \in L^\infty(D)$ with $\varphi_\xi(\bx)\in\mathbb{R}$ is said to be \emph{elliptic} if
  \begin{align}
    \exists(\theta_\XI,\xi_-)\in\mathbb{R}\times\mathbb{R}_{>0}, \ \
    \text{a.e.~in $D$}, \ \
    |\xi(\bx)| \cos(\theta_\xi+\varphi_\xi(\bx))\geq\xi_-.
    \label{eqn:ellipticitySca}
  \end{align}
  In addition, we use the notation $\xi_+ := \|\xi\|_{L^\infty(D)}$.
\end{definition}

The next proposition shows that a scalar field is elliptic if and only if all its values, a.e.~in $D$, belong to a fixed closed half-place of $\mathbb{C}$ that does not contain zero.
\begin{proposition}
  \label{prop:ellipticScalar}
  Let a scalar field $\xi(\bx) = |\xi(\bx)|\:e^{\im\varphi_\xi(\bx)} \in L^\infty(D)$ with $\varphi_\xi(\bx)\in\mathbb{R}$.
  The field $\xi$ is elliptic if and only if $\xi_{\inf} := \inf_{\bx\in D} |\xi(\bx)| > 0$ and there exist $\theta_\xi\in\mathbb{R}$ and $\eta_\xi\in{]0,\pi/2]}$ such that, a.e.~in $D$,
  \begin{align}
    \theta_\xi+\varphi_\xi(\bx) \ \in \ \bigcup_{n\in\mathbb{Z}} \big[2n\pi - (\pi/2-\eta_\xi) \ , \ 2n\pi + (\pi/2-\eta_\xi)\big].
    \label{eqn:ellipticity:sca}
  \end{align}
  In addition, when $\xi$ is elliptic, $\eta_\xi$ can be chosen arbitrarily close to $0$.
\end{proposition}
\begin{proof}
  By using Definition~\ref{def:ellipticitySca}, the reverse implication of the proposition is straightforward.
  For the direct implication, we have $\xi_+ < +\infty$ because $\xi(\bx) \in L^\infty(D)$.
  Then, Equation~\eqref{eqn:ellipticitySca} gives, a.e., $\cos(\theta_\xi+\varphi_\xi(\bx))\geq\xi_-/\xi_+ > 0$.
  Therefore, Equation~\eqref{eqn:ellipticity:sca} holds with $\eta_\xi = \pi/2 - \arccos(\xi_-/\xi_+) = \arcsin(\xi_-/\xi_+) \in{]0,\pi/2]}$.

  When $\xi$ is elliptic, let $(\theta_\xi,\xi_-)$ be an admissible pair in Equation~\eqref{eqn:ellipticitySca}.
  We can then choose $\eta_\xi = \arcsin(\xi_-/\xi_+)$.
  Nevertheless, $(\theta_\xi,\xi_-/2^k)$ is also an admissible pair for every $k\in\mathbb{N}$, and we can also choose $\eta_\xi = \eta_\xi^{(k)} := \arcsin(\xi_-/\xi_+/2^k)$.
  Because $\lim_{k\rightarrow\infty} \eta_\xi^{(k)} = 0$, $\eta_\xi$ can then be arbitrarily close to $0$.
\end{proof}

The ellipticity directions and the $\Theta$-set can be explicitly identified in certain cases.

\begin{proposition}
  \label{prop:ellipticScalar:corr1}
  Let a scalar field $\xi(\bx) = |\xi(\bx)|\:e^{\im\varphi_\xi(\bx)} \in L^\infty(D)$ with $\varphi_\xi(\bx) \in {[\varphi_{-},\varphi_{+}]}\subset\mathbb{R}$ a.e.~in $D$.
  If $\inf_{\bx\in D} |\xi(\bx)| > 0$ and $\varphi_{+}-\varphi_{-}<\pi$, then the field $\xi$ is elliptic and every angle
  \begin{align}
    \theta_\xi \ \in \ \bigcup_{n\in\mathbb{Z}} {\big] 2n\pi - \pi/2 - \varphi_- \ , \ 2n\pi + \pi/2 - \varphi_+ \big[}
    \label{eqn:ellipticity:sca2}
  \end{align}
  is an ellipticity direction.
\end{proposition}
\begin{proof}
  By Proposition~\ref{prop:ellipticScalar}, the field $\xi$ is elliptic if and only if there exist $\theta_\xi\in\mathbb{R}$ and $\eta_\xi\in{]0,\pi/2]}$ such that, for almost every $\vec{x}\in D$, there exists $m\in\mathbb{Z}$ such that
  \begin{align}
    2m\pi-(\pi/2-\eta_\xi) \leq \theta_\xi + \varphi_\xi(\bx) \leq 2m\pi+(\pi/2-\eta_\xi).
    \label{eqn:rangeEllSca:tmp1}
  \end{align}
  This condition is necessarily satisfied if $\theta_\xi$ satisfies
  \begin{align}
    2m\pi - (\pi/2-\eta_\xi) - \varphi_- \leq \theta_\xi \leq 2m\pi + (\pi/2-\eta_\xi) - \varphi_+.
  \end{align}
  Because $\eta_\xi$ can be arbitrarily close to $0$ (by Proposition~\ref{prop:ellipticScalar}), this relation holds for every angle
  \begin{align}
    \theta_{\xi} \  \in \  {\big] 2m\pi - \pi/2 - \varphi_- \ , \ 2m\pi + \pi/2 - \varphi_+ \big[}.
    \label{eqn:rangeEllSca:tmp2}
  \end{align}
  This range is not empty because $\varphi_+ - \varphi_- < \pi$.
  If $\theta_\xi=\theta_0$ satisfies \eqref{eqn:rangeEllSca:tmp1} for a given $m=n_0$, then every $\theta_\xi=\theta_0+2n\pi$ with $n\in\mathbb{Z}$ also satisfies \eqref{eqn:rangeEllSca:tmp1} for $m=n_0+n$.
  Therefore, the condition is satisfied for every angle in the range \eqref{eqn:rangeEllSca:tmp2} up to a multiple of $2\pi$.
  This holds a.e.~in $D$.
\end{proof}

\begin{proposition}
  \label{prop:ellipticScalar:corr2}
  If, in addition to the hypotheses of Prop.~\ref{prop:ellipticScalar:corr1}, we have $\varphi_{-} = \min_{\bx\in D} \varphi_\xi(\bx)$ and $\varphi_{+} = \max_{\bx\in D} \varphi_\xi(\bx)$, then there is no ellipticity direction outside the range \eqref{eqn:ellipticity:sca2}.
\end{proposition}

\vspace{-1.5em}
\begin{proof}
  Because the ellipticity directions are defined up to a multiple of $2\pi$, it is sufficient to prove that there is no ellipticity direction in the range $[\pi/2-\varphi_+, 3\pi/2-\varphi_-]$.

  We proceed by contradiction.
  We assume that $\theta_\xi$ is an ellipticity direction that satisfies
  \begin{align}
    \pi/2-\varphi_+ \leq \theta_\xi
    \qquad\text{and}\qquad
    \theta_\xi \leq 3\pi/2-\varphi_-.
    \label{eqn:rangeEllSca:tmp7}
  \end{align}
  Returning to the proof of Proposition \ref{prop:ellipticScalar:corr1}, the angle $\theta_\xi$ must satisfy the condition \eqref{eqn:rangeEllSca:tmp1} in particular for every $\vec{x}\in D$ such that $\varphi_\xi(\vec{x})=\varphi_-$ or $\varphi_\xi(\vec{x})=\varphi_+$.

  If $\varphi_\xi(\vec{x})=\varphi_-$, this condition is satisfied if and only if there exists $m_-\in\mathbb{Z}$ such that
  \begin{align}
    2m_-\pi - (\pi/2 - \eta_-) - \varphi_- \leq \theta_\xi
    \qquad\text{and}\qquad
    \theta_\xi \leq 2m_-\pi + (\pi/2 - \eta_-) - \varphi_-,
    \label{eqn:rangeEllSca:tmp8}
  \end{align}
  for some $\eta_-\in{]0,\pi/2]}$ that can be arbitrarily close to $0$.
  Combining these inequalities with the second and first inequalities of \eqref{eqn:rangeEllSca:tmp7}, respectively, yields
  \begin{align}
    \eta_- \leq 2(1-m_-)\pi
    \qquad\text{and}\qquad
    \eta_- \leq 2m_-\pi + (\varphi_+ - \varphi_-).
  \end{align}
  These conditions are satisfied only if $m_- = 0$ and $\varphi_- < \varphi_+$.
  Then, \eqref{eqn:rangeEllSca:tmp8} gives
  \begin{align}
    \theta_\xi \ \in \ \big[-\pi/2 - \varphi_- + \eta_- \ , \ \pi/2 - \varphi_- - \eta_-\big].
    \label{eqn:rangeTmpMinus}
  \end{align}

  We proceed similarly for the case $\varphi_\xi(\vec{x})=\varphi_+$.
  Then, the condition \eqref{eqn:rangeEllSca:tmp1} is satisfied if and only if there exists $m_+\in\mathbb{Z}$ such that
  \begin{align}
    2m_+\pi - (\pi/2 - \eta_+) - \varphi_+ \leq \theta_\xi
    \qquad\text{and}\qquad
    \theta_\xi \leq 2m_+\pi + (\pi/2 - \eta_+) - \varphi_+,
    \label{eqn:rangeEllSca:tmp9}
  \end{align}
  for some $\eta_+\in{]0,\pi/2]}$ that can be arbitrarily close to $0$.
  Again, combining these inequalities with the second and first inequalities of \eqref{eqn:rangeEllSca:tmp7}, respectively, gives
  \begin{align}
    \eta_+ \leq 2(1-m_+)\pi + (\varphi_+ - \varphi_-)
    \qquad\text{and}\qquad
    \eta_+ \leq 2m_+\pi.
  \end{align}
  These conditions are satisfied only if $m_+ = 1$ and $\varphi_- < \varphi_+$.
  Then, \eqref{eqn:rangeEllSca:tmp9} gives
  \begin{align}
    \theta_\xi \ \in \ \big[3\pi/2 - \varphi_+ + \eta_+ \ , \ 5\pi/2 - \varphi_+ - \eta_+\big].
    \label{eqn:rangeTmpPlus}
  \end{align}

  Every angle $\theta_\xi$ verifying the conditions \eqref{eqn:rangeEllSca:tmp7} must belong to the ranges \eqref{eqn:rangeTmpMinus} and \eqref{eqn:rangeTmpPlus} to be an ellipticity direction.
  However, the intersection of both ranges is empty.
  Indeed, we have $\pi/2 - \varphi_- - \eta_- < 3\pi/2 - \varphi_+ + \eta_+$, which is equivalent to $\varphi_+ - \varphi_- < \pi + \eta_- + \eta_+$, which is always true because $\varphi_+ - \varphi_- < \pi$ by hypothesis and $\eta_- + \eta_+$ is strictly positive.
  Therefore, there is no ellipticity direction in the range $[\pi/2-\varphi_+, 3\pi/2-\varphi_-]$.
  %
\end{proof}

\begin{corollary}
  \label{prop:ellipticScalar:corr3}
  In addition to the hypotheses of Proposition \ref{prop:ellipticScalar:corr1},
  \begin{itemize}
    \item if $\varphi_- \in {]-3\pi/2,\pi/2]}$ and $\varphi_+\in[-\pi/2,3\pi/2[$, then
          \begin{align}
            \Theta_\xi \ \supseteq \
            {]-\pi/2 - \varphi_- , \pi/2 - \varphi_+[},
          \end{align}
    \item if $\varphi_- \in {]-3\pi/2,-\pi/2[}$ and $\varphi_+\in{]-3\pi/2,-\pi/2[}$, then
          \begin{align}
            \Theta_\xi \ \supseteq \
            {[-\pi , -3\pi/2 - \varphi_+[}
            \ \cup \
            {]-\pi/2 - \varphi_- , \pi]},
          \end{align}
    \item if $\varphi_- \in {]\pi/2,3\pi/2[}$ and $\varphi_+\in{]\pi/2,3\pi/2[}$, then
          \begin{align}
            \Theta_\xi \ \supseteq \
            {[-\pi , \pi/2 - \varphi_+[}
            \ \cup \
            {]3\pi/2 - \varphi_- , \pi]}.
          \end{align}
  \end{itemize}
  The equalities hold if $\varphi_{-} = \min_{\bx\in D} \varphi_\xi(\bx)$ and $\varphi_{+} = \max_{\bx\in D} \varphi_\xi(\bx)$.
\end{corollary}
\begin{proof}
  Direct by using Propositions \ref{prop:ellipticScalar:corr1} and \ref{prop:ellipticScalar:corr2}.
\end{proof}

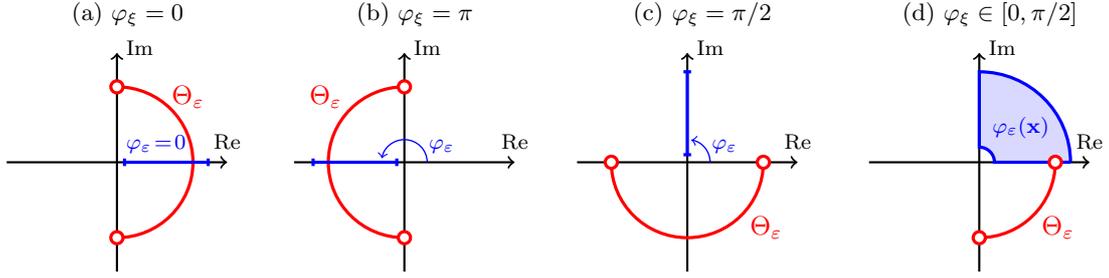
\begin{figure}[tb]
  \centering
  \begin{subfigure}[b]{0.24\textwidth}
    \centering
    \caption{$\varphi_\xi=0$}
    \label{fig:elliptScalarField:a}
    \begin{tikzpicture}
      \draw[thick, ->] (-1.45,0) -- (1.45,0);
      \draw[thick, ->] (0,-1.45) -- (0,1.45);
      \draw (1.15,0.05) node [above right] {\footnotesize$\mathrm{Re}$};
      \draw (0,1.3) node [above right] {\footnotesize$\mathrm{Im}$};
      \draw[color=red, line width=1.2pt] (0,-1) arc (-90:90:1);
      \draw[color=red, line width=1.2pt, fill=white] (0,1) circle (0.08);
      \draw[color=red, line width=1.2pt, fill=white] (0,-1) circle (0.08);
      \draw[color=blue, line width=1.2pt] (0.1,0) -- (1.2,0);
      \draw[color=blue, line width=1.2pt, -] (0.1,-0.05) -- (0.1,0.05);
      \draw[color=blue, line width=1.2pt, -] (1.2,-0.05) -- (1.2,0.05);
      \draw (0.6,0.6) node [above right] {\color{red}$\Theta_\varepsilon$};
      \draw (0,0) node [above right] {\footnotesize\color{blue}$\varphi_\varepsilon\!=\!0$};
    \end{tikzpicture}
  \end{subfigure}
  \hfill
  \begin{subfigure}[b]{0.24\textwidth}
    \centering
    \caption{$\varphi_\xi=\pi$}
    \label{fig:elliptScalarField:b}
    \begin{tikzpicture}
      \draw[thick, ->] (-1.45,0) -- (1.45,0);
      \draw[thick, ->] (0,-1.45) -- (0,1.45);
      \draw (1.15,0.05) node [above right] {\footnotesize$\mathrm{Re}$};
      \draw (0,1.3) node [above right] {\footnotesize$\mathrm{Im}$};
      \draw[color=red, line width=1.2pt] (0,1) arc (90:270:1);
      \draw[color=red, line width=1.2pt, fill=white] (0,1) circle (0.08);
      \draw[color=red, line width=1.2pt, fill=white] (0,-1) circle (0.08);
      \draw[color=blue, line width=1.2pt] (-0.1,0) -- (-1.2,0);
      \draw[color=blue, line width=1.2pt, -] (-0.1,-0.05) -- (-0.1,0.05);
      \draw[color=blue, line width=1.2pt, -] (-1.2,-0.05) -- (-1.2,0.05);
      \draw[color=blue, line width=0.5pt, ->] (0.3,0) arc (0:170:0.3);
      \draw (-0.7,0.6) node [above left] {\color{red}$\Theta_\varepsilon$};
      \draw (0.2,0) node [above right] {\footnotesize\color{blue}$\varphi_\varepsilon$};
    \end{tikzpicture}
  \end{subfigure}
  \hfill
  \begin{subfigure}[b]{0.24\textwidth}
    \centering
    \caption{$\varphi_\xi=\pi/2$}
    \begin{tikzpicture}
      \draw[thick, ->] (-1.45,0) -- (1.45,0);
      \draw[thick, ->] (0,-1.45) -- (0,1.45);
      \draw (1.15,0.05) node [above right] {\footnotesize$\mathrm{Re}$};
      \draw (0,1.3) node [above right] {\footnotesize$\mathrm{Im}$};
      \draw[color=red, line width=1.2pt] (-1,0) arc (180:360:1);
      \draw[color=red, line width=1.2pt, fill=white] (-1,0) circle (0.08);
      \draw[color=red, line width=1.2pt, fill=white] (1,0) circle (0.08);
      \draw[color=blue, line width=1.2pt] (0,0.1) -- (0,1.2);
      \draw[color=blue, line width=1.2pt, -] (-0.05,0.1) -- (0.05,0.1);
      \draw[color=blue, line width=1.2pt, -] (-0.05,1.2) -- (0.05,1.2);
      \draw[color=blue, line width=0.5pt, ->] (0.3,0) arc (0:80:0.3);
      \draw (0.7,-0.6) node [below right] {\color{red}$\Theta_\varepsilon$};
      \draw (0.2,0) node [above right] {\footnotesize\color{blue}$\varphi_\varepsilon$};
    \end{tikzpicture}
    \label{fig:elliptScalarField:c}
  \end{subfigure}
  \hfill
  \begin{subfigure}[b]{0.24\textwidth}
    \centering
    \caption{$\varphi_\xi\in{[0,\pi/2]}$}
    \label{fig:elliptScalarField:d}
    \begin{tikzpicture}
      \draw[color=blue!15, line width=0pt, fill=blue!15] (0,0) --  (90:1.2) arc(90:0:1.2) -- cycle;
      \draw[color=white, line width=0pt, fill=white] (0,0) --  (90:0.2) arc(90:0:0.2) -- cycle;
      \draw[thick, ->] (-1.45,0) -- (1.45,0);
      \draw[thick, ->] (0,-1.45) -- (0,1.45);
      \draw (1.15,0.05) node [above right] {\footnotesize$\mathrm{Re}$};
      \draw (0,1.3) node [above right] {\footnotesize$\mathrm{Im}$};
      \draw[color=blue, line width=1.2pt] (1.2,0) -- (0.2,0) arc (0:90:0.2) -- (0,0.2) -- (0,1.2) -- (0,1.2) arc (90:0:1.2);
      \draw[color=red, line width=1.2pt] (0,-1) arc (270:360:1);
      \draw[color=red, line width=1.2pt, fill=white] (0,-1) circle (0.08);
      \draw[color=red, line width=1.2pt, fill=white] (1,0) circle (0.08);
      \draw (0.7,-0.6) node [below right] {\color{red}$\Theta_\varepsilon$};
      \draw (0.05,0.2) node [above right] {\footnotesize\color{blue}$\varphi_\varepsilon(\bx)$};
    \end{tikzpicture}
  \end{subfigure}
  \caption{Representation of the set of all the ellipticity directions $\Theta_\xi$ (in red) for a scalar field $\xi(\bx) = |\xi(\bx)|\:e^{\im\varphi_\xi(\bx)}$ (in blue) with $\inf_{\bx\in D}|\xi(\bx)|>0$ and $\sup_{\bx\in D}|\xi(\bx)|<+\infty$ in several configurations.
    On figure \ref{fig:elliptScalarField:d}, we assume that $\min_{\bx\in D}\varphi_\varepsilon(\bx)=0$ and $\max_{\bx\in D}\varphi_\varepsilon(\bx)=\pi/2$.
  }
  \label{fig:elliptScalarField}
\end{figure}

Several configurations are illustrated in Figure~\ref{fig:elliptScalarField}.
If $\xi(\bx)$ is purely real, it is elliptic if and only if it is either positive a.e.~(then $\Theta_\xi={]-\pi/2,\pi/2[}$, see Figure~\ref{fig:elliptScalarField:a}) or negative a.e.~(then $\Theta_\xi={[-\pi,-\pi/2[}\cup{]\pi/2,\pi]}$, see Figure~\ref{fig:elliptScalarField:b}).
Similarly, if $\xi(\bx)$ is purely imaginary, it is elliptic if and only if the imaginary part is either positive a.e.~(then $\Theta_\xi={]-\pi,0[}$, see Figure~\ref{fig:elliptScalarField:c}) or negative a.e.~(then $\Theta_\xi={]0,\pi[}$).
For a general field $\xi(\bx) = |\xi(\bx)|\:e^{\im\varphi_\xi(\bx)}$ with $\varphi_\xi(\bx)\in{[0,\pi/2]}$ (Figure~\ref{fig:elliptScalarField:d}), we have $\Theta_\xi\supseteq{]-\pi/2,0[}$.
In addition, if $\min_{\bx\in D}|\varphi_\xi(\bx)|=0$ and $\max_{\bx\in D}|\varphi_\xi(\bx)|=\pi/2$, then there is no ellipticity direction outside $]-\pi/2,0[$, and the equality holds.


\subsection{Properties of elliptic tensor fields}


\subsubsection{Normal and Hermitian tensor fields}

For normal tensor fields, the general ellipticity condition \eqref{eqn:ellipticityTens} can be rewritten by using the eigenvalues of the tensor fields at every point.

Let $\XI\in\tens{L}^\infty(D)$ be a normal $d$-by-$d$ tensor field, and let $\bx\in D$ be such that $\XI(\bx)$ is well-defined.
Because $\XI(\bx)$ is normal, it is diagonalizable and there exists an orthonormal basis of $\mathbb{C}^d$ composed of eigenvectors of $\XI(\bx)$.
Below, the eigenvalues and associated orthonormal eigenvectors of $\XI$ at $\bx\in D$ are denoted by $\{\lambda_i(\bx)\}_{i=1\dots d}\in\mathbb{C}$ and $\{\bv_i(\bx)\}_{i=1\dots d}\in\mathbb{C}^d$, respectively.
Each eigenvalue is represented as $\lambda_i(\bx) = |\lambda_i(\bx)| e^{\im\beta_i(\bx)}$ with $\beta_i(\bx)\in\mathbb{R}$.
If, in addition, $\XI(\bx)$ is Hermitian, all the eigenvalues are real.

In the normal case, for every point $\bx\in D$ and every vector $\bz\in\mathbb{C}^d$, we have that $(\XI(\bx)\bz)\cdot\overline{\bz} = \sum_i \lambda_i(\bx)|z_i|^2$ with $\bz = \sum_i z_i\bv_i(\bx)$ and $z_i := \bz\cdot\overline{\bv_i}(\bx)$ for $i=1\dots d$.
Using this property, Definition~\ref{def:ellipticityTens} applied to normal tensor fields can be rewritten as follows.
\begin{definition}[Normal elliptic tensor field]
  \label{def:ellipticityTensNorm}
  The normal $d$-by-$d$ tensor field $\XI\in \tens{L}^\infty(D)$ is said to be \emph{elliptic} if
  \begin{align}
    \exists(\theta_\XI,\xi_-)\in\mathbb{R}\times\mathbb{R}_{>0}, \ \
    \text{a.e.~in $D$}, \ \
    |\lambda_i(\bx)| \cos(\theta_\XI+\beta_i(\bx)) \geq \xi_-,
    \ \ \text{for } i=1\dots d,
  \end{align}
  where $\{\lambda_i(\bx)\}_{i=1\dots d}$ are the eigenvalues of $\XI(\bx)$ represented as $\lambda_i(\bx) = |\lambda_i(\bx)| e^{\im\beta_i(\bx)}$ with $\beta_i(\bx)\in\mathbb{R}$ for $i=1\dots d$.
\end{definition}
\noindent
Therefore, the tensor $\XI$ is elliptic if and only if there exists a pair $(\theta_\XI,\xi_-)$ for which all the eigenvalues, a.e., satisfy the ellipticity condition.

Propositions \ref{prop:ellipticScalar}, \ref{prop:ellipticScalar:corr1} and \ref{prop:ellipticScalar:corr2} can be straightforwardly extended to normal tensor fields.
The main difference is that all the eigenvalues of $\XI(\bx)$, a.e., must satisfy the hypotheses fulfilled by the elliptic scalar fields.
\begin{proposition}[Normal elliptic tensor field]
  \label{prop:ellipticityTensNorm}
  Let $\XI\in\tens{L}^\infty(D)$ be a normal $d$-by-$d$ tensor field.
  For all $\bx\in D$, let $\{\lambda_i(\bx)\}_{i=1\dots d}\in\mathbb{C}$ be the eigenvalues of $\XI(\bx)$ represented as $\lambda_i(\bx) = |\lambda_i(\bx)| e^{\im\beta_i(\bx)}$ with $\beta_i(\bx)\in\mathbb{R}$ for $i=1\dots d$.
  \vspace{-1em}
  \begin{enumerate}
    \item[(a)]
          The tensor field $\XI$ is elliptic if and only if $\inf_{\bx\in D} |\XI(\bx)|>0$ and there exist $\theta_\XI\in\mathbb{R}$ and $\eta_\XI\in{]0,\pi/2]}$ such that, a.e.~in $D$,
          \begin{align}
            \theta_\xi + \beta_i(\bx) \ \in \ \bigcup_{n\in\mathbb{Z}} {\big[ 2n\pi - (\pi/2 - \eta_\XI) \ , \ 2n\pi + (\pi/2 - \eta_\XI) \big]},
            \ \ \text{for } i=1\dots d.
          \end{align}
          In addition, when $\XI$ is elliptic, $\eta_\XI$ can be chosen arbitrarily close to $0$.
    \item[(b)]
          If $\inf_{\bx\in D}|\XI(\bx)|>0$ and $\{\beta_i(\bx)\}_{i=1\dots d} \in [\beta_-,\beta_+]\subset\mathbb{R}$ a.e.~in $D$ with $\beta_+-\beta_-<\pi$, then $\XI$ is elliptic and every angle
          \begin{align}
            \theta_\xi \ \in \ \bigcup_{n\in\mathbb{Z}} {\big] 2n\pi - \pi/2 - \beta_- \ , \ 2n\pi + \pi/2 - \beta_+ \big[}
            \label{eqn:ellipticity:norm2}
          \end{align}
          is an ellipticity direction.
    \item[(c)]
          If, in addition to the hypotheses of (b), there exist $i_-,i_+=1\dots d$ and $\bx_-,\bx_+\in D$ such that $\beta_{i_-}(\bx_-) = \beta_-$ and $\beta_{i_+}(\bx_+) = \beta_+$, then there is no ellipticity direction outside the range \eqref{eqn:ellipticity:norm2}.
  \end{enumerate}
\end{proposition}
\begin{proof}
  The proofs of (a), (b) and (c) are similar to the proofs of Propositions \ref{prop:ellipticScalar}, \ref{prop:ellipticScalar:corr1} and \ref{prop:ellipticScalar:corr2}, respectively.
\end{proof}

In the Hermitian case, all the eigenvalues are real.
As a consequence, a Hermitian tensor field is elliptic if and only if all its eigenvalues are bounded and either uniformly positive or uniformly negative a.e.
\begin{proposition}[Hermitian elliptic tensor field]
  Let $\XI\in\tens{L}^\infty(D)$ be a Hermitian d-by-d tensor field.
  The tensor field $\XI$ is elliptic if and only if all the eigenvalues of $\XI$, denoted $\{\lambda_i\}_{i=1\dots d}$, are strictly positive or strictly negative, and there exist $\xi_+,\xi_-\in\mathbb{R}_{>0}$ such that, a.e., $|\lambda_i(\bx)|\in{[\xi_-,\xi_+]}$ for $i=1\dots d$.
  Moreover, $\Theta_\XI={]-\pi/2,\pi/2[}$ if the eigenvalues are positive, and $\Theta_\XI={[-\pi,-\pi/2[}\cup{]\pi/2,\pi]}$ if they are negative.
\end{proposition}
\begin{proof}
  Direct.
\end{proof}


\subsubsection{General tensor fields}

For general complex tensor fields, we have the following results.

\begin{proposition}[Inverse of an elliptic tensor field]
  If the tensor field $\ZETA\in\tens{L}^\infty(D)$ is elliptic, then there exists an inverse tensor field $\RHO:=\ZETA^{-1}$ that belongs to $\tens{L}^\infty(D)$ with $\rho_+ \in {[\zeta_+^{-1},\zeta_-^{-1}]}$.
  Moreover, for any $\theta_\ZETA\in\Theta_{\ZETA}$, this tensor field satisfies the ellipticity condition with $\theta_{\RHO} = -\theta_{\ZETA}$ and $\rho_- = \zeta_-\zeta_+^{-2}$.
  In addition, $\Theta_{\RHO} = -\Theta_{\ZETA}$.
\end{proposition}
\vspace{-4mm}
\begin{proof}
  Let $\bx\in D$ be such that $\ZETA(\bx)$ is well-defined.
  Because of the ellipticity condition (i.e.~Equation \eqref{eqn:ellipticityTens} for $\XI=\ZETA$), we have $\ZETA(\bx)\bz=\bzero$ if and only if $\bz=\bzero$.
  Therefore, $\ZETA(\bx)$ is injective, and there exists an inverse tensor $\RHO(\bx)$ such that $\ZETA\RHO$ is the identity tensor.

  Let $\by\in\mathbb{C}^d$ and $\bz = \RHO\by$.
  Because $\ZETA\in\tens{L}^\infty(D)$, we have $|\ZETA\bz| \leq \zeta_+|\bz|$, and also $\zeta_+^{-1}|\by| \leq |\RHO\by|$ a.e.~in $D$.
  Because $\ZETA$ is elliptic, we have $\zeta_- |\bz|^2 \leq \Re\{e^{\im\theta_\ZETA}\:(\ZETA\bz)\cdot\overline{\bz}\} \leq |(\ZETA\bz)\cdot\overline{\bz}| \leq |\bz\|\ZETA\bz|$, then $\zeta_- |\bz| \leq |\ZETA\bz|$, and also $|\RHO\by| \leq \zeta_-^{-1}|\by|$ a.e.~in $D$.
  Therefore, $\zeta_+^{-1} \leq \|\RHO\|_{\tens{L}^\infty(D)} \leq \zeta_-^{-1}$ and $\RHO\in\tens{L}^\infty(D)$.

  By using again the ellipticity condition, together with
  $\Re\{e^{i\theta_\ZETA}\:(\ZETA\bz)\cdot\overline{\bz}\}
    = \Re\{e^{i\theta_\ZETA}\:\by\cdot\overline{(\RHO\by)}\}
    = \Re\{e^{-i\theta_\ZETA}\:(\RHO\by)\cdot\overline{\by}\}$
  and $\zeta_- |\bz|^2 = \zeta_- |\RHO\by|^2 \geq \zeta_- \zeta_+^{-2} |\by|^2$, we obtain that
  \begin{align}
    \exists(\theta_\ZETA,\zeta_-)\in\mathbb{R}\times\mathbb{R}_{>0}, \ \
    \text{a.e.~in $D$}, \ \
    \forall\by\in\mathbb{C}^d, \ \
    \Re\{e^{-i\theta_\ZETA}\:(\RHO\by)\cdot\overline{\by}\} \geq \zeta_- \zeta_+^{-2} |\by|^2.
  \end{align}
  Therefore, $\RHO$ is elliptic with $(\theta_\RHO,\rho_-) = (-\theta_\ZETA,\zeta_-\zeta_+^{-2})$, and $\Theta_{\RHO} \supset -\Theta_{\ZETA}$.
  Because $\ZETA = \RHO^{-1}$, we also have $\Theta_{\ZETA} \supset -\Theta_{\RHO}$, and then $\Theta_{\RHO} = -\Theta_{\ZETA}$.
\end{proof}

\begin{proposition}[Scaling of an elliptic tensor field]
  If the tensor field $\ZETA\in\tens{L}^\infty(D)$ is elliptic, then the tensor field $\RHO:=\alpha\ZETA$ with $\alpha = |\alpha| e^{\im\beta} \in \mathbb{C}\backslash\{0\}$ belongs to $\tens{L}^\infty(D)$ with $\rho_+ = |\alpha|\zeta_+$.
  Moreover, for any $\theta_\ZETA\in\Theta_{\ZETA}$, this tensor field satisfies the ellipticity condition with $\theta_{\RHO} = \theta_{\ZETA}-\beta$ and $\rho_- = |\alpha|\zeta_-$.
\end{proposition}
\vspace{-4mm}
\begin{proof}
  Direct.
\end{proof}

\begin{proposition}[Eigenvalues of an elliptic tensor field]
  If the tensor field $\ZETA\in\tens{L}^\infty(D)$ is elliptic, then all the eigenvalues of $\ZETA(\bx)$ a.e.~in $D$ belong to a fixed closed half-plane of $\mathbb{C}$ that does not contain zero.
\end{proposition}
\vspace{-4mm}
\begin{proof}
  Let $(\theta_\ZETA,\zeta_-)\in\mathbb{R}\times\mathbb{R}_{>0}$ be a pair such that the ellipticity condition is satisfied for $\ZETA$.
  Let $\lambda\in\mathbb{C}$ be an eigenvalue of $\ZETA(\bx)$ for a given point $\bx\in D$.
  Then, there exists an eigenvector $\bv\in\mathbb{C}^d\backslash\{0\}$ such that $\ZETA(\bx)\bv = \lambda\bv$.
  Because $\ZETA\in\tens{L}^\infty(D)$, we have $\zeta_+|\bv| \geq |\ZETA(\bx)\bv|$, that is $\zeta_+ \geq |\lambda|$.
  By using the ellipticity condition, we have $\Re\{e^{\im\theta_\ZETA}\:(\ZETA(\bx)\bv)\cdot\overline{\bv}\} \geq \zeta_-|\bv|^2$, and then $\Re\{e^{\im\theta_\ZETA}\:\lambda\} \geq \zeta_-$.
  In particular, $|\lambda|\geq\zeta_->0$.
  Denoting $\lambda=|\lambda| e^{\im\beta}$, we have $\cos(\theta_\zeta+\beta) \geq \zeta_-/\zeta_+ > 0$.
  Therefore,
  \begin{align}
    \theta_\zeta + \beta \ \in \ \bigcup_{n\in\mathbb{Z}} \big[2 n\pi - (\pi/2-\eta_\zeta) \ ,\  2n\pi + (\pi/2-\eta_\zeta) \big]
  \end{align}
  with $\eta_\zeta = \pi/2-\mathrm{arccos}(\zeta_-/\zeta_+) \in {]0,\pi/2]}$.
  Since the set is the same for all the eigenvalues of $\ZETA(\bx)$ a.e.~in $D$, the result holds.
\end{proof}

\begin{proposition}
  If the tensor field $\ZETA\in\tens{L}^\infty(D)$ is elliptic, then
  \begin{align}
    \forall\bv\in\BL^2(D), \ \
    \forall\theta_\ZETA\in\Theta_\ZETA, \ \
    \zeta_- \|\bv\|_{\BL^2(D)}^2
    \leq \Re\big\{e^{\im\theta_\ZETA} (\ZETA\bv,\bv)_D\big\}
    \leq \big|(\ZETA\bv,\bv)_D\big|
    \leq \zeta_+ \|\bv\|_{\BL^2(D)}^2.
  \end{align}
\end{proposition}
\vspace{-4mm}
\begin{proof}
  Direct.
\end{proof}
\section{Time-harmonic electromagnetic problems}
\label{sec:problems}

In this section, we consider the $\be$-formulation of general time-harmonic electromagnetic problems written with equation \eqref{eqn:MaxwellEqn} in a domain $\Omega$, and completed with a boundary condition on $\Gamma:=\partial\Omega$.
We consider problems with only one type of boundary condition, which is a Dirichlet condition \eqref{eqn:DirichetBC}, a Neumann condition \eqref{eqn:NeumannBC} or a Robin condition \eqref{eqn:NeumannBC}.
The problems are called Dirichlet, Neumann and Robin problems, respectively.

After a presentation of the geometric and functional framework, we define the variational formulations, and we review general well-posedness results based on \cite{chicaud2021phd, chicaud2021analysis, chicaud2023analysis}.
Then, we study conditions such that the sesquilinear forms of the variational formulations are coercive.


\subsection{Geometric and functional framework}

We make the following geometric hypotheses.
\begin{hypothesis}[Geometry \texttt{I}]
  \label{hyp:geom:1}
  The domain $\Omega$ is an open, connected and bounded subset of $\mathbb{R}^3$ with a Lipschitz-continuous boundary $\Gamma$.
\end{hypothesis}
\begin{hypothesis}[Geometry \texttt{II}]
  \label{hyp:geom:2}
  The boundary $\Gamma$ is of class $\mathcal{C}^2$ or is polyhedral without pathological vertex (for a definition, see \cite{buffa2002boundary}).
\end{hypothesis}

The \emph{tangential trace operator} and \emph{tangential components trace operator} are denoted $\gamma^T : \bv \rightarrow (\bv\times\bn)_{|\Gamma}$ and $\pi^T : \bv \rightarrow \bn\times(\bv\times\bn)_{|\Gamma}$.
Depending on the considered problem, the solution belongs to one of the following spaces
\begin{align}
  \BH(\curl,\Omega)   & := \{\bv\in\BL^2(\Omega), \: \curl\bv\in\BL^2(\Omega)\},          \\
  \BH_0(\curl,\Omega) & := \{\bv\in\BH(\curl,\Omega), \: \gamma^T\bv=0\},                 \\
  \BH^+(\curl,\Omega) & := \{\bv\in\BH(\curl,\Omega), \: \gamma^T\bv\in\BL_t^2(\Gamma)\},
\end{align}
where $\BL_t^{2}(\Gamma) := \{\bv\in\BL^2(\Gamma), \: \bv\cdot\bn=0 \: \text{ a.e.~on } \Gamma\}$.
Endowed with the norm
\begin{align}
  \left\|\bv\right\|_{\BH(\curl,\Omega)}
  := \sqrt{
  \left\|\bv\right\|^2_{\BL^2(\Omega)}
  + \left\|\Curl\bv\right\|^2_{\BL^2(\Omega)}
  },
\end{align}
the spaces $\BH(\curl,\Omega)$ and $\BH_0(\curl,\Omega)$ are Hilbert spaces.
Endowed with the norm
\begin{align}
  \left\|\bv\right\|_{\BH^+(\curl,\Omega)}
  := \sqrt{
  \left\|\bv\right\|^2_{\BL^2(\Omega)}
  + \left\|\Curl\bv\right\|^2_{\BL^2(\Omega)}
  + \left\|\gamma^T\bv\right\|^2_{\BL^2(\Gamma)}
  },
\end{align}
the space $\BH^+(\curl,\Omega)$ is a Hilbert space.
%
%
We introduce the \emph{surface operators}
\begin{align}
  \nabla_\Gamma(v_{|\Gamma}) & := \pi^T(\nabla v),                          \\
  \curl_\Gamma(v_{|\Gamma})  & := \gamma^T(\nabla v),                       \\
  \sdiv_\Gamma(\gamma^T\bv)  & := \text{dual operator of $-\nabla_\Gamma$}, \\
  \scurl_\Gamma(\pi^T\bv)    & := \text{dual operator of $\curl_\Gamma$},
\end{align}
and the \emph{surface spaces}
\begin{align}
  \BH_\parallel^{1/2}(\Gamma)               & := \pi^T(\BH^1(\Omega)),                                                                     \\
  \BH_\perp^{1/2}(\Gamma)                   & := \gamma^T(\BH^1(\Omega)),                                                                  \\
  \BH_\parallel^{-1/2}(\sdiv_\Gamma,\Gamma) & := \{\bv\in\big[\BH_\parallel^{1/2}(\Gamma)\big]', \: \sdiv_\Gamma\bv\in H^{-1/2}(\Gamma)\}, \\
  \BH_\perp^{-1/2}(\scurl_\Gamma,\Gamma)    & := \{\bv\in\big[\BH_\perp^{1/2}(\Gamma)\big]', \: \scurl_\Gamma\bv\in H^{-1/2}(\Gamma)\}.
\end{align}
For a Lipschitz boundary $\Gamma$, the trace operators $\gamma^T$ and $\pi^T$ are continuous and surjective from $\BH(\curl,\Omega)$ to the surface spaces $\BH_\parallel^{-1/2}(\sdiv_\Gamma,\Gamma)$ and $\BH_\perp^{-1/2}(\scurl_\Gamma,\Gamma)$, respectively, see~\cite{buffa2002traces}.
In addition, these surface spaces are dual spaces, and we have the classical integration by parts relation
\begin{align}
  (\bu,\curl\bv)_\Omega - (\curl\bu,\bv)_\Omega & = \langle\gamma^T\bu,\pi^T\bv\rangle_\Gamma, \quad\forall\bu,\bv\in\BH(\curl,\Omega),
\end{align}
where $(\ba,\bb)_\Omega=\int_\Omega\ba\cdot\overline{\bb}\:d\bx$, and $\langle\cdot,\cdot\rangle_\Gamma$ denotes the duality product between $\BH_\parallel^{-1/2}(\sdiv_\Gamma,\Gamma)$ and $\BH_\perp^{-1/2}(\scurl_\Gamma,\Gamma)$.


\subsection{Variational formulations}
\label{sec:globalPbm:form}

We make the following material hypotheses.
\begin{hypothesis}[Medium \texttt{I}]
  \label{hyp:mat:1}
  The electric permittivity $\EPS$ and the magnetic permeability $\MU$ are elliptic complex 3-by-3 tensor fields belonging to $\tens{L}^\infty(\Omega)$.
\end{hypothesis}
\begin{hypothesis}[Medium \texttt{II}]
  \label{hyp:mat:2}
  The impedance $\ALPHA$ is an elliptic complex 2-by-2 tensor field belonging to $\tens{L}_t^\infty(\Gamma)$.
  It verifies additional regularity hypotheses given in \cite[Section 4]{chicaud2023analysis}.
\end{hypothesis}

The variational formulations of the Dirichlet, Neumann and Robin problems are written as follows.
Under Hypotheses \ref{hyp:geom:1}, \ref{hyp:geom:2}, \ref{hyp:mat:1} and \ref{hyp:mat:2} (for the Robin case), these variational formulations are equivalent to the corresponding differential formulations.
See \cite[Theorem 5.1]{chicaud2023analysis} for the Robin case.

\begin{problem}[Dirichlet problem]
\label{pbm:dirichlet:globalVF}
Find $\be = \be_0+\be_\mathrm{lift}$ with $\be_0\in\BH_0(\curl,\Omega)$ such that
\begin{align}
  a_\D(\be_0,\bv) = \ell_\D(\bv),
  \quad\forall\bv\in\BH_0(\curl,\Omega),
\end{align}
where the sesquilinear form $a_\D$ and the antilinear form $\ell_\D$ are defined as
\begin{align}
  a_\D(\bu,\bv)
   & := (\MU^{-1}\curl\bu,\curl\bv)_\Omega
  - \omega^2 (\EPS\bu,\bv)_\Omega,
  \label{eqn:dirichlet:sesqform}
  \\
  \ell_\D(\bv)
   & := (\bf + \omega^2 \EPS\be_\mathrm{lift},\bv)_\Omega
  - (\MU^{-1}\curl\be_\mathrm{lift},\curl\bv)_\Omega, \nonumber
\end{align}
where $\be_\mathrm{lift}\in\BH(\curl,\Omega)$ is a lift function such that $\be_\mathrm{lift}\times\bn = \bg_\D$, for given data $\bf\in\BL^2(\Omega)$ and $\bg_\D\in\BH^{-1/2}_\parallel(\sdiv_\Gamma,\Gamma)$.
\end{problem}

\begin{problem}[Neumann problem]
\label{pbm:neumann:globalVF}
Find $\be\in\BH(\curl,\Omega)$ such that
\begin{align}
  a_\N(\be,\bv) = \ell_\N(\bv),
  \quad\forall\bv\in\BH(\curl,\Omega),
\end{align}
where the sesquilinear form $a_\N$ and the antilinear form $\ell_\N$ are defined as
\begin{align}
  a_\N(\bu,\bv)
   & := (\MU^{-1}\curl\bu,\curl\bv)_\Omega
  - \omega^2 (\EPS\bu,\bv)_\Omega,
  \label{eqn:neumann:sesqform}
  \\
  \ell_\N(\bv)
   & := (\bf,\bv)_\Omega
  + \langle\bg_\N,\pi^T\bv\rangle_\Gamma, \nonumber
\end{align}
for given data $\bf\in\BL^2(\Omega)$ and $\bg_\N\in\BH^{-1/2}_\perp(\scurl_\Gamma,\Gamma)$.
\end{problem}

\begin{problem}[Robin problem]
\label{pbm:robin:globalVF}
Find $\be\in\BH^+(\curl,\Omega)$ such that
\begin{align}
  a_\R(\be,\bv) = \ell_\R(\bv),
  \quad\forall\bv\in\BH^+(\curl,\Omega),
\end{align}
where the sesquilinear form $a_\R$ and the antilinear form $\ell_\R$ are defined as
\begin{align}
  a_\R(\bu,\bv)
   & := (\MU^{-1}\curl\bu,\curl\bv)_\Omega
  - \omega^2 (\EPS\bu,\bv)_\Omega
  - \langle\ALPHA\gamma^T\bu,\gamma^T\bv\rangle_\Gamma,
  \label{eqn:robin:sesqform}
  \\
  \ell_\R(\bv)
   & := (\bf,\bv)_\Omega
  - \langle\bg_\R,\gamma^T\bv\rangle_\Gamma, \nonumber
\end{align}
for given data $\bf\in\BL^2(\Omega)$ and $\bg_\R\in\BH^{-1/2}_\parallel(\sdiv_\Gamma,\Gamma) \cap \BH^{-1/2}_\perp(\scurl_\Gamma,\Gamma)$.
\end{problem}

Let us note that the Robin boundary condition regularizes the solution at the boundary, see e.g.~\cite{chicaud2023analysis}.
This ensures that $\gamma^T\be\in\BL^2_t(\Gamma)$ and $\pi^T(\MU^{-1}\curl\be)\in\BL^2_t(\Gamma)$, which justifies the choice of the solution space $\BH^+(\curl,\Omega)$.


\subsection{Well-posedness of the problems}
\label{sec:globalPbm:wellposed}

Time-harmonic electromagnetic problems can exhibit resonance phenomena, which mathematically correspond to non-unique solutions, while we generally seek a unique solution.
Therefore, two definitions of the well-posedness are generally used.
To simplify the presentation, we introduce the following abstract problem:
\begin{align}
  \left|\ \text{Find $\bu\in\BV$ such that $a(\bu,\bv)=\ell(\bv)$ for all $\bv\in\BV$,}\right.
  \label{eqn:genPbm}
\end{align}
with a Hilbert space $(\BV,\|\cdot\|_\BV)$, a continuous sesquilinear form $a:\BV\times\BV\rightarrow\mathbb{C}$ and a continuous antilinear form $\ell:\BV\rightarrow\mathbb{C}$.

\vspace{2mm}

\begin{definition}[Well-posedness in the Hadamard sense]
  The problem \eqref{eqn:genPbm} is \emph{well-posed in the Hadamard sense} if it admits a unique solution $\bu\in\BV$ which depends continuously on the data, i.e.~there exists $C\in\mathbb{R}_{>0}$ such that $\|\bu\|_\BV \leq C\:\|\ell\|_{\BV'}$.
\end{definition}

\begin{definition}[Well-posedness in the Fredholm sense]
  The problem \eqref{eqn:genPbm} is \emph{well-posed in the Fredholm sense} if either (1) the problem admits a unique solution $\bu\in\BV$ which depends continuously on the data, or (2) the problem has solutions if and only if the data satisfy a finite number of compatibility conditions.
  In the second case, the solutions form an affine space of finite dimension, and the component of the solution that is $\BV$-orthogonal to the corresponding linear space depends continuously on the data.
\end{definition}

The Dirichlet and Neumann problems (i.e.~Problems \ref{pbm:dirichlet:globalVF} and \ref{pbm:neumann:globalVF}) are well-posed in the Fredholm sense under the hypotheses \ref{hyp:geom:1}, \ref{hyp:geom:2} and \ref{hyp:mat:1}, see \cite[Theorem 5.3]{chicaud2021analysis} and \cite[Theorem 4.2]{chicaud2021analysis}.
For the Robin problem (i.e.~Problem \ref{pbm:robin:globalVF}), the well-posedness is proven in \cite[Theorem 5.3]{chicaud2023analysis} under the additional hypothesis \ref{hyp:mat:1} and the following additional hypothesis.
\begin{hypothesis}[Medium \texttt{III}]
  \label{hyp:mat:3}
  The tensor fields $-\MU^{-1}$ and $\ALPHA$ are simultaneously elliptic.
\end{hypothesis}

If a problem is well-posed in the Fredholm sense, the uniqueness of the solution implies the well-posedness in the Hadamard sense.
The uniqueness of the solution can be proven for the Robin problem if the tensor fields $\EPS$ and $\MU$ are real symmetric positive definite.
It is a consequence of a unique continuation principle~\cite{vogelsang1991strong}.
Unfortunately, to the best of our knowledge, no similar result is available for more general tensor fields.

On the other hand, by the Lax-Milgram lemma, the problem \eqref{eqn:genPbm} is well-posed in the Hadamard sense if the sesquilinear form $a$ is coercive.
In addition, the continuous dependency on the data holds with the constant $C=C_{\mathrm{coe}}^{-1}$, where $C_{\mathrm{coe}}$ is the coercivity constant defined below.
\begin{definition}[Coercivity]
  A sesquilinear form $a:\BV\times\BV\rightarrow\mathbb{C}$ is \emph{coercive} if there exist $C_{\mathrm{coe}}\in\mathbb{R}_{>0}$ and $\theta_{\mathrm{coe}}\in\mathbb{R}$ such that
  \begin{align}
    \Re\{-e^{\im\theta_{\mathrm{coe}}}a(\bu,\bu)\} \geq C_{\mathrm{coe}}\|\bu\|_{\BV}^2,
    \quad \forall\bu\in\BV.
    \label{eqn:coercivity}
  \end{align}
\end{definition}
The coercivity of the sesquilinear forms of the considered problems is investigated in the next section.



\subsection{Coercivity of the sesquilinear forms}

Under an additional hypothesis on the material fields $\EPS$ and $-\MU^{-1}$, we can prove that the sesquilinear forms of the Dirichlet, Neumann and Robin problems are coercive.
\begin{hypothesis}[Medium \texttt{IV}]
  \label{hyp:mat:4}
  The tensor fields $\EPS$ and $-\MU^{-1}$ are simultaneously elliptic.
\end{hypothesis}
\begin{proposition}[Coercivity; Dirichlet and Neumann cases]
  \label{thm:wellpos:coercive:DirNeu}
  Assume that Hypotheses~\ref{hyp:mat:1} and \ref{hyp:mat:4} hold.
  Let $\Theta$ be the intersection of the $\Theta$-sets of $\EPS$ and $-\MU^{-1}$.
  Then, the sesquilinear forms \eqref{eqn:dirichlet:sesqform} and \eqref{eqn:neumann:sesqform} are coercive for all $\theta_{\mathrm{coe}}\in\Theta$ with the coercivity constant $C_{\mathrm{coe}} = \min(\mu_-\mu_+^{-2},\omega^2 \varepsilon_-)$, which depends on $\theta_{\mathrm{coe}}$.
\end{proposition}
\begin{proof}
  Let a given $\theta\in\Theta$.
  Because $\EPS$ and $-\MU^{-1}$ are elliptic for the ellipticity direction $\theta$, there exist strictly positive real constants $\varepsilon_-$ and $\mu_-$ such that, a.e., for all $\bz_\varepsilon\in\mathbb{C}^3$ and $\bz_\mu\in\mathbb{C}^3$,
  \begin{align}
    \Re\{e^{\im\theta}(\EPS\bz_\varepsilon)\cdot\overline{\bz_\varepsilon}\}
    \geq \varepsilon_- |\bz_\varepsilon|^2
    \quad\text{and}\quad
    -\Re\{e^{\im\theta}(\MU^{-1}\bz_\mu)\cdot\overline{\bz_\mu}\}
    \geq \mu_-\mu_+^{-2} |\bz_\mu|^2
  \end{align}
  with $\mu_+ := \|\MU\|_{\BL^\infty(\Omega)}$.
  Then, for all $\bu\in\BH(\curl,\Omega)$, we have
  \begin{align}
    \Re\big\{e^{\im\theta} \omega^2 (\EPS\bu,\bu)_\Omega\big\}
     & \geq \omega^2 \varepsilon_- \left\|\bu\right\|^2_{\BL^2(\Omega)}, \\
    -\Re\big\{e^{\im\theta}(\MU^{-1}\curl\bu,\curl\bu)_\Omega\big\}
     & \geq \mu_-\mu_+^{-2} \left\|\curl\bu\right\|^2_{\BL^2(\Omega)},
  \end{align}
  and then
  \begin{align}
    \Re\big\{-e^{\im\theta}a(\bu,\bu)\big\}
     & \geq \mu_-\mu_+^{-2} \left\|\curl\bu\right\|^2_{\BL^2(\Omega)}
    + \omega^2 \varepsilon_- \left\|\bu\right\|^2_{\BL^2(\Omega)}                                  \\
     & \geq \min(\mu_-\mu_+^{-2},\omega^2 \varepsilon_-) \left\|\bu\right\|^2_{\BH(\Curl,\Omega)}.
  \end{align}
  Therefore, the coercivity condition holds for every $\theta_{\mathrm{coe}}=\theta\in\Theta$.
\end{proof}

The following general result is adapted from \cite{chicaud2023analysis}.
%

\begin{proposition}[Coercivity; Robin case]
  \label{thm:wellpos:coercive}
  Assume that Hypotheses~\ref{hyp:mat:1}, \ref{hyp:mat:2}, \ref{hyp:mat:3} and \ref{hyp:mat:4} hold.
  Let $\Theta$ be the intersection of the $\Theta$-sets of $\EPS$, $-\MU^{-1}$ and $\ALPHA$.
  Then, the sesquilinear form \eqref{eqn:robin:sesqform} are coercive for all $\theta_{\mathrm{coe}}\in\Theta$ with the coercivity constant $C_{\mathrm{coe}} = \min(\mu_-\mu_+^{-2},\omega^2 \varepsilon_-,\alpha_-)$, which depends on $\theta_{\mathrm{coe}}$.
\end{proposition}
\begin{proof}
  Let a given $\theta\in\Theta$.
  Because $\EPS$, $-\MU^{-1}$ and $\ALPHA$ are elliptic for the ellipticity direction $\theta$, there exist strictly positive real constants $\varepsilon_-$, $\mu_-$ and $\alpha_-$ such that, a.e., for all $\bz_\varepsilon\in\mathbb{C}^3$, $\bz_\mu\in\mathbb{C}^3$ and $\bz_\alpha\in\mathbb{C}^2$,
  \begin{align}
    \Re\{e^{\im\theta}(\EPS\bz_\varepsilon)\cdot\overline{\bz_\varepsilon}\}
     & \geq \varepsilon_- |\bz_\varepsilon|^2,
    \\
    -\Re\{e^{\im\theta}(\MU^{-1}\bz_\mu)\cdot\overline{\bz_\mu}\}
     & \geq \mu_-\mu_+^{-2} |\bz_\mu|^2,
    \\
    \Re\{e^{\im\theta}(\ALPHA\bz_\alpha)\cdot\overline{\bz_\alpha}\}
     & \geq \alpha_- |\bz_\alpha|^2,
  \end{align}
  with $\mu_+ := \|\MU\|_{\BL^\infty(\Omega)}$.
  Then, for all $\bu\in\BH^+(\curl,\Omega)$, we have
  \begin{align}
    \Re\big\{e^{\im\theta} \omega^2 (\EPS\bu,\bu)_\Omega\big\}
     & \geq \omega^2 \varepsilon_- \left\|\bu\right\|^2_{\BL^2(\Omega)}, \\
    -\Re\big\{e^{\im\theta}(\MU^{-1}\curl\bu,\curl\bu)_\Omega\big\}
     & \geq \mu_-\mu_+^{-2} \left\|\curl\bu\right\|^2_{\BL^2(\Omega)},   \\
    \Re\big\{e^{\im\theta} (\ALPHA\gamma^T\bu,\gamma^T\bu)_\Gamma\big\}
     & \geq \alpha_- \left\|\bu\right\|^2_{\BL^2_t(\Gamma)},
  \end{align}
  and then
  \begin{align}
    \Re\big\{-e^{\im\theta}a(\bu,\bu)\big\}
     & \geq \mu_-\mu_+^{-2} \left\|\curl\bu\right\|^2_{\BL^2(\Omega)}
    + \omega^2 \varepsilon_- \left\|\bu\right\|^2_{\BL^2(\Omega)}
    + \alpha_- \left\|\bu\right\|^2_{\BL^2_t(\Gamma)}                                                         \\
     & \geq \min(\mu_-\mu_+^{-2},\omega^2 \varepsilon_-,\alpha_-) \left\|\bu\right\|^2_{\BH^+(\Curl,\Omega)}.
  \end{align}
  Therefore, the coercivity condition holds for every $\theta_{\mathrm{coe}}=\theta\in\Theta$.
\end{proof}
\section{Examples of problems with elliptic material coefficients}
\label{sec:applications}

In this section, we discuss time-harmonic electromagnetic problems with material media used in the literature.
We consider isotropic lossy media, geometric media (e.g.~media used in perfectly matched layers), and gyrotropic media (e.g.~media used in cold plasma models).


\subsection{Isotropic lossy media}

In isotropic lossy media, the electric permittivity and magnetic permeability tensor fields can be written as $\EPS = \varepsilon_c\tens{I}_3$ and $\MU = \mu\tens{I}_3$, respectively, with the \emph{complex dielectric field} $\varepsilon_c(\vec{x}) := \varepsilon(\vec{x}) + \im\sigma(\vec{x})/\omega\in\mathbb{C}$, where the coefficients $\varepsilon(\vec{x})$, $\mu(\vec{x})$ and $\sigma(\vec{x})$ are strictly positive real fields.
For the Robin problem, we consider an isotropic impedance $\ALPHA = \alpha\tens{I}_2$ with a \emph{complex impedance field} $\alpha(\vec{x})\in\mathbb{C}$.
We assume that all the scalar fields are bounded.

We can prove that the sesquilinear form of the Dirichlet, Neumann and Robin problems are coercive under some hypotheses.
For sufficiently smooth fields $\bu$ and $\bv$, the sesquilinear form of the Dirichlet and Neumann problems can be rewritten as
\begin{align}
    a(\bu,\bv)
    = (\mu^{-1}\curl\bu,\curl\bv)_\Omega
    - \omega^2 (\varepsilon_c\bu,\bv)_\Omega,
    \label{eqn:isotrop:dirichlet:sesqform}
\end{align}
and the one of the Robin problem reads
\begin{align}
    a^+(\bu,\bv)
    = a(\bu,\bv)
    - (\alpha \gamma^T\bu,\gamma^T\bv)_\Gamma.
    \label{eqn:isotrop:robin:sesqform}
\end{align}
Below, the complex fields $\varepsilon_c(\vec{x})$ and $\alpha(\vec{x})$ are written in polar form as
\begin{align}
    \varepsilon_c(\vec{x}) = |\varepsilon_c(\vec{x})|\:e^{\im\varphi_{\varepsilon}(\vec{x})} \quad\text{and}\quad
    \alpha(\vec{x}) = |\alpha(\vec{x})|\:e^{\im\varphi_\alpha(\vec{x})},
\end{align}
with $\varphi_{\varepsilon}(\vec{x})\in{]0,\pi/2[}$ and $\varphi_{\alpha}(\vec{x}) \in \mathbb{R}$.

\begin{proposition}[Coercivity; Dirichlet and Neuman cases]
    \label{thm:coercive:dir:lossyScalar}
    If $\varphi_{\varepsilon}(\vec{x}) \in {[\varphi_{\varepsilon,-},\varphi_{\varepsilon,+}]}$ a.e.~with $0 < \varphi_{\varepsilon,-} \leq \varphi_{\varepsilon,+} < \pi/2$, then the sesquilinear form \eqref{eqn:isotrop:dirichlet:sesqform} of the Dirichlet and Neumann problems is coercive for all coercivity direction $\theta_\mathrm{coe}\in\Theta_\mathrm{coe}$ with $\Theta_\mathrm{coe} \supseteq \left]-\pi/2-\varphi_{\varepsilon,-} , -\pi/2\right[$.
    In addition, the equality holds if $\varphi_{\varepsilon,-}=\min_{\bx\in D}\varphi_{\varepsilon}(\vec{x})$.
\end{proposition}
\begin{proof}
    If $\varepsilon_c$ and $-\mu^{-1}$ are simultaneously elliptic, i.e.~the intersection of their $\Theta$-sets is not empty, then the sesquilinear form is coercive thanks to Proposition~\ref{thm:wellpos:coercive:DirNeu},

    By Corollary \ref{prop:ellipticScalar:corr3}, $\varepsilon_c$ is an elliptic field and its $\Theta$-set is $\Theta_{\varepsilon_c} \supseteq {]- \pi/2 - \varphi_{\varepsilon,-} , \pi/2 - \varphi_{\varepsilon,+}[}$.
    Because $\mu$ is a strictly positive elliptic real field, the field $-\mu^{-1}$ is a strictly negative elliptic real field and its $\Theta$-set is $\Theta_{-\mu^{-1}} = {[-\pi,-\pi/2[} \cup {]\pi/2,\pi]}$.
    We directly have that the intersection of both sets satisfies
    \begin{align}
        \Theta_{\varepsilon_c}\cap\Theta_{-\mu^{-1}} \ \supseteq \ \left]-\pi/2-\varphi_{\varepsilon,-},-\pi/2\right[,
    \end{align}
    and the equality holds if $\varphi_{\varepsilon,-}$ is the minimum value of $\varphi_{\varepsilon}(\vec{x})$.
\end{proof}
The $\Theta$-sets used in the proof are illustrated in Figure \ref{fig:elliptDir:lossyMedium:a}.

\begin{figure}[tb]
    \centering
    \begin{subfigure}[b]{0.45\textwidth}
        \centering
        \caption{Fields $\varepsilon_c$ and $-\mu^{-1}$}
        \label{fig:elliptDir:lossyMedium:a}
        \begin{tikzpicture}[scale=1.2]
            \draw[thick, ->] (-2.45,0) -- (2.45,0);
            \draw[thick, ->] (0,-2.45) -- (0,2.45);
            \draw (2.1,0.05) node [above right] {\footnotesize$\mathrm{Re}$};
            \draw (0,2.3) node [above right] {\footnotesize$\mathrm{Im}$};

            \draw[color=blue, line width=0.5pt, densely dotted] (75:0) -- (75:2.4);
            \draw[color=blue, line width=0.5pt, densely dotted] (30:0) -- (30:2.4);
            \draw[color=red, line width=0.5pt, densely dotted] (-120:0) -- (-120:1.5);
            \draw[color=red, line width=0.5pt, densely dotted] (15:0) -- (15:1.3);

            \draw[color=blue, line width=1pt, fill=blue!15] (75:2.2) arc (75:30:2.2) -- (30:0.3) arc (30:75:0.3) -- cycle;
            \draw (0.35,0.75) node [above right] {\scriptsize\color{blue}$\varepsilon_c(\bx)$};

            \draw[color=blue, line width=1.2pt] (-0.3,0) -- (-2.2,0);
            \draw[color=blue, line width=1.2pt, -] (-0.3,-0.1) -- (-0.3,0.1);
            \draw[color=blue, line width=1.2pt, -] (-2.2,-0.1) -- (-2.2,0.1);
            \draw (-1.0,0) node [above left] {\scriptsize\color{blue}$-\mu^{-1}(\bx)$};

            \draw[color=red, line width=1pt] (-120:0.8) arc (-120:15:0.8);
            \draw[color=red, line width=1pt, fill=white] (-120:0.8) circle (0.05);
            \draw[color=red, line width=1pt, fill=white] (15:0.8) circle (0.05);

            \draw[color=red, line width=1pt] (90:1) arc (90:270:1);
            \draw[color=red, line width=1pt, fill=white] (90:1) circle (0.05);
            \draw[color=red, line width=1pt, fill=white] (270:1) circle (0.05);


            \draw (-45:0.7) node [below right] {\color{red}\scriptsize$\Theta_{\varepsilon_c}$};
            \draw (-150:0.8) node [below left] {\color{red}\scriptsize$\Theta_{-\mu^{-1}}$};

            \draw[color=black!50!white, line width=0.5pt, ->] (0:1.6) arc (0:30:1.6);
            \draw (10:1.5) node [above right] {\scriptsize\color{black!50!white}$\beta_1$};
            \draw[color=black!50!white, line width=0.5pt, <-] (75:1.6) arc (75:90:1.6);
            \draw (81:1.6) node [above] {\scriptsize\color{black!50!white}$\beta_2$};
            \draw[color=black!50!white, line width=0.5pt, ->] (0:1.1) arc (0:15:1.1);
            \draw (-4:1.05) node [above right] {\scriptsize\color{black!50!white}$\beta_2$};
            \draw[color=black!50!white, line width=0.5pt, ->] (-90:1.2) arc (-90:-120:1.2);
            \draw (-105:1.2) node [below] {\scriptsize\color{black!50!white}$\beta_1$};

            \draw (0.3,-1.2) node [below right] {\color{black!50!white}\fbox{\scriptsize\color{black!50!white}
                    $\begin{aligned}
                            \beta_1 & =\varphi_{\varepsilon,-}       \\
                            \beta_2 & =\pi/2-\varphi_{\varepsilon,+}
                        \end{aligned}$}};
        \end{tikzpicture}
    \end{subfigure}
    \qquad
    \begin{subfigure}[b]{0.45\textwidth}
        \centering
        \caption{Field $\alpha$}
        \label{fig:elliptDir:lossyMedium:b}
        \begin{tikzpicture}[scale=1.2]

            \draw[color=blue, line width=0.5pt, densely dotted] (190:0) -- (190:2.4);
            \draw[color=blue, line width=0.5pt, densely dotted] (25:0) -- (25:2.4);
            \draw[color=blue, line width=1pt, fill=blue!15] (190:1.7) arc (190:25:1.7) -- (25:0.6) arc (25:190:0.6) -- cycle;
            \draw (0.25,0.65) node [above right] {\scriptsize\color{blue}$\alpha(\bx)$};

            \draw[thick, ->] (-2.45,0) -- (2.45,0);
            \draw[thick, ->] (0,-2.45) -- (0,2.45);
            \draw (2.1,0.05) node [above right] {\footnotesize$\mathrm{Re}$};
            \draw (0,2.3) node [above right] {\footnotesize$\mathrm{Im}$};

            \draw[color=red, line width=0.5pt, densely dotted] (-115:0) -- (-115:1.8);
            \draw[color=red, line width=0.5pt, densely dotted] (-100:0) -- (-100:1.8);

            \draw[color=red, line width=1pt] (-115:0.8) arc (-115:-100:0.8);
            \draw[color=red, line width=1pt, fill=white] (-115:0.8) circle (0.05);
            \draw[color=red, line width=1pt, fill=white] (-100:0.8) circle (0.05);
            \draw (-115:0.8) node [left] {\color{red}\scriptsize$\Theta_{\alpha}$};

            \draw[color=black!50!white, line width=0.5pt, ->] (0:1.5) arc (0:25:1.5);
            \draw (13:1.5) node [right] {\scriptsize\color{black!50!white}$\eta_1$};
            \draw[color=black!50!white, line width=0.5pt, ->] (180:1.8) arc (180:190:1.8);
            \draw (186:1.8) node [left] {\scriptsize\color{black!50!white}$\eta_2$};
            \draw[color=black!50!white, line width=0.5pt, ->] (-90:1.1) arc (-90:-115:1.1);
            \draw (-108:1.1) node [below] {\scriptsize\color{black!50!white}$\eta_1$};
            \draw[color=black!50!white, line width=0.5pt, ->] (-90:1.5) arc (-90:-100:1.5);
            \draw (-95:1.5) node [below] {\scriptsize\color{black!50!white}$\eta_2$};

            \draw (0.3,-1.2) node [below right] {\color{black!50!white}\fbox{\scriptsize\color{black!50!white}
                    $\begin{aligned}
                            \eta_1 & =\varphi_{\alpha,-}     \\
                            \eta_2 & =\varphi_{\alpha,+}-\pi
                        \end{aligned}$}};
        \end{tikzpicture}
    \end{subfigure}
    \caption{Representation of the $\Theta$-sets (in red) on the complex plane for the fields $\varepsilon_c(\bx)$, $-\mu^{-1}(\bx)$ and $\alpha(\bx)$ (in blue) for a non-homogeneous isotropic lossy medium.
        The fields have values in regions of the complex plan that are colored blue.}
    \label{fig:elliptDir:lossyMedium}
\end{figure}
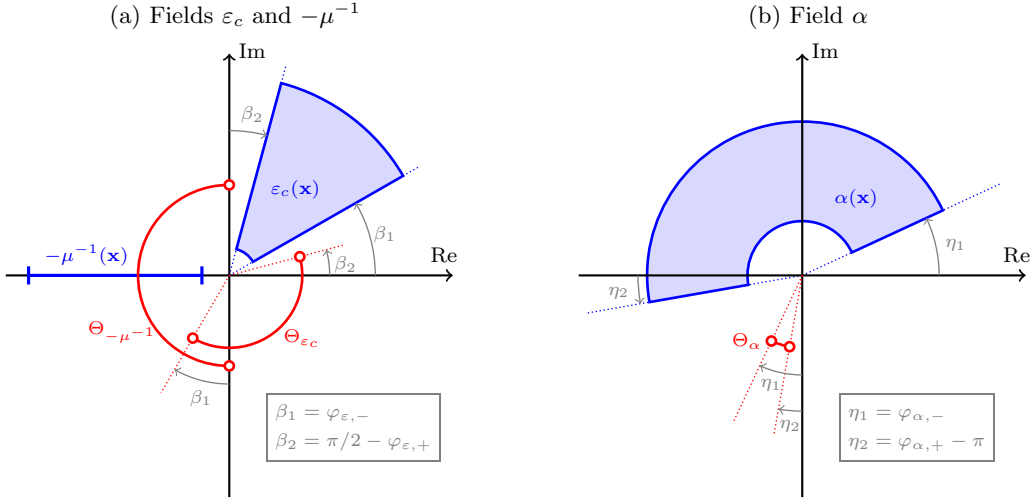

\begin{proposition}[Coercivity; Robin case]
    \label{thm:coercive:rob:lossyScalar}
    If $\varphi_{\varepsilon}(\vec{x}) \in {[\varphi_{\varepsilon,-},\varphi_{\varepsilon,+}]}$ and $\varphi_{\alpha}(\vec{x}) \in {[\varphi_{\alpha,-},\varphi_{\alpha,+}]}$ a.e.~with
    $0 < \varphi_{\varepsilon,-} \leq \varphi_{\varepsilon,+} < \pi/2$,
    $0 < \varphi_{\alpha,-} \leq \varphi_{\alpha,+} < \pi+\varphi_{\varepsilon,-}$ and
    $\varphi_{\alpha,+} - \varphi_{\alpha,-} < \pi$,
    then the sesquilinear form \eqref{eqn:isotrop:robin:sesqform} of the Robin problem is coercive for all coercivity direction $\theta_\mathrm{coe}\in\Theta_\mathrm{coe}$ with
    \begin{align}
        \Theta_\mathrm{coe} \ \supseteq \ \big]-\pi/2-\min(\varphi_{\varepsilon,-},\varphi_{\alpha,-}), -\pi/2 - \max(0,\varphi_{\alpha,+}-\pi)\big[.
    \end{align}
    In addition, the equality holds if $\varphi_{\varepsilon,-}=\min_{\bx\in D}\varphi_{\varepsilon}(\vec{x})$, $\varphi_{\alpha,-}=\min_{\bx\in D}\varphi_{\alpha}(\vec{x})$ and $\varphi_{\alpha,+}=\max_{\bx\in D}\varphi_{\alpha}(\vec{x})$.
\end{proposition}
\begin{proof}
    If $\varepsilon_c$, $-\mu^{-1}$ and $\alpha$ are simultaneously elliptic, i.e.~the intersection of the $\Theta$-sets of $\varepsilon_c$, $-\mu^{-1}$ and $\alpha$ is not empty, then the sesquilinear form is coercive thanks to Proposition~\ref{thm:wellpos:coercive}.
    Following the proof of Proposition~\ref{thm:coercive:rob:lossyScalar}, we already have that
    \begin{align}
        \Theta_{\varepsilon}\cap\Theta_{-\mu^{-1}}
        \ \supseteq \
        {]-\pi/2-\varphi_{\varepsilon,-},-\pi/2[},
        \label{eqn:intersection}
    \end{align}
    where the equality holds if $\varphi_{\varepsilon,-}=\min_{\bx\in D}\varphi_{\varepsilon}(\vec{x})$.

    By Proposition \ref{prop:ellipticScalar:corr1}, if $\varphi_{\alpha,-} \leq \varphi_{\alpha,+}$ and $\varphi_{\alpha,+} - \varphi_{\alpha,-} < \pi$, the scalar field $\alpha$ is elliptic, and every angle
    \begin{align}
        \theta_\alpha \ \in \ \bigcup_{n\in\mathbb{Z}} {\big] 2n\pi - \pi/2 - \varphi_{\alpha,-} \ , \ 2n\pi + \pi/2 - \varphi_{\alpha,+} \big[}
    \end{align}
    is an ellipticity direction of $\alpha$.
    For a given $n$, the intersection between the set \eqref{eqn:intersection} and the interval ${] 2n\pi - \pi/2 - \varphi_{\alpha,-} , 2n\pi + \pi/2 - \varphi_{\alpha,+} [}$ is not empty if and only if
    \begin{align}
        2n\pi < \varphi_{\alpha,-}
        \qquad\text{and}\qquad
        \varphi_{\alpha,+} < 2n\pi + \pi + \varphi_{\varepsilon,-}.
    \end{align}

    In particular, since $0 < \varphi_{\alpha,-}$ and $\varphi_{\alpha,+} < \pi + \varphi_{\varepsilon,-}$, the intersection is not empty for $n=0$.
    Therefore, under these conditions, the sesquilinear form is coercive.

    Let us characterize more precisely the set of the coercivity directions.
    If $\varphi_{\alpha,-} \in {]0,\pi/2]}$ and $\varphi_{\alpha,+}\in{]0,\pi + \varphi_{\varepsilon,-}[}$, then, by Corollary \ref{prop:ellipticScalar:corr3},
    \begin{align}
        \Theta_\alpha
        \ \supseteq \
        {]-\pi/2 - \varphi_{\alpha,-} , \pi/2 - \varphi_{\alpha,+}[}
    \end{align}
    and the intersection between all the $\Theta$-sets is such that
    \begin{align}
        \Theta_{\varepsilon}\cap\Theta_{-\mu^{-1}}\cap\Theta_\alpha
        \ \supseteq \
        {]-\pi/2-\min(\varphi_{\varepsilon,-},\varphi_{\alpha,-}) , \min(-\pi/2,\pi/2 - \varphi_{\alpha,+})[}.
    \end{align}
    On the other hand, if $\varphi_{\alpha,-} \in {]\pi/2,\pi + \varphi_{\varepsilon,-}[}$ and $\varphi_{\alpha,+}\in{]\pi/2,\pi + \varphi_{\varepsilon,-}[}$, then, by Corollary \ref{prop:ellipticScalar:corr3},
    \begin{align}
        \Theta_\alpha
        \ \supseteq \
        {[-\pi , \pi/2 - \varphi_{\alpha,+}[}
        \ \cup \
        {]3\pi/2 - \varphi_{\alpha,-} , \pi]}
    \end{align}
    and the intersection is such that
    \begin{align}
        \Theta_{\varepsilon}\cap\Theta_{-\mu^{-1}}\cap\Theta_\alpha
        \ \supseteq \
        {]-\pi/2-\varphi_{\varepsilon,-},\min(-\pi/2,\pi/2 - \varphi_{\alpha,+})[}.
    \end{align}
    The equality holds if $\varphi_{\varepsilon,-}=\min_{\bx\in D}\varphi_{\varepsilon}(\vec{x})$, and if $\varphi_{\alpha,-}$ and $\varphi_{\alpha,+}$ are the minimal and maximal values of $\varphi_{\alpha}(\vec{x})$, respectively.
    Since $\min(-\pi/2,\pi/2 - \varphi_{\alpha,+}) = -\pi/2 - \max(0,\varphi_{\alpha,+}-\pi)$, the result holds.
\end{proof}
The $\Theta$-sets used in the proof are illustrated in Figure \ref{fig:elliptDir:lossyMedium}.

Coercivity constants are easily obtained for cases with constant material properties.
\begin{corollary}[Coercivity; Dirichlet and Neumann case]
    In addition to the assumptions of Proposition~\ref{thm:coercive:dir:lossyScalar}, the fields $\varepsilon_c$ and $\mu$ are constant, then, for all coercivity direction $\theta_\mathrm{coe}\in\Theta_\mathrm{coe}$, the coercivity constant of the sesquilinear form $a$ can be chosen equal to
    \begin{align}
        C_\mathrm{coe} =
        \min\left( - \mu^{-1} \cos(\theta_\mathrm{coe}), \
        \omega^2 |\varepsilon_c| \cos(\theta_\mathrm{coe}+\varphi_{\varepsilon_c})
        \right).
    \end{align}
\end{corollary}
\begin{proof}
    For all $\theta_\mathrm{coe}\in\Theta_\mathrm{coe}$, the sesquilinear form \eqref{eqn:isotrop:dirichlet:sesqform} gives successively
    \begin{align}
        \Re\{a(-e^{\im\theta_\mathrm{coe}}\bu,\bu)\}
         & = - \Re\{e^{\im\theta_\mathrm{coe}} (\mu^{-1} \curl\bu,\curl\bu)_\Omega\}
        + \omega^2 \Re\{e^{\im\theta_\mathrm{coe}} (\varepsilon_c \bu,\bu)_\Omega\}                                           \\
         & = - \Re\{\mu^{-1} e^{\im\theta_\mathrm{coe}}\} \left\|\curl\bu\right\|_{\BL^2(\Omega)}^2
        + \omega^2 \Re\{|\varepsilon_c| e^{\im\varphi_{\varepsilon_c}} e^{\im\theta_\mathrm{coe}}\} \|\bu\|_{\BL^2(\Omega)}^2 \\
         & =
        - \mu^{-1} \cos(\theta_\mathrm{coe}) \left\|\curl\bu\right\|_{\BL^2(\Omega)}^2
        + \omega^2 |\varepsilon_c| \cos(\theta_\mathrm{coe}+\varphi_{\varepsilon_c}) \left\|\bu\right\|_{\BL^2(\Omega)}^2     \\
         & \geq C_\mathrm{coe} \left\|\bu\right\|_{\BH(\curl,\Omega)}^2,
    \end{align}
    which gives the result.
    By direct inspection, we get $\cos(\theta_\mathrm{coe})<0$ and $\cos(\theta_\mathrm{coe}+\varphi_{\varepsilon_c})>0$, and then $C_\mathrm{coe}>0$.
\end{proof}

\begin{corollary}[Coercivity; Robin case]
    In addition to the assumptions of Proposition~\ref{thm:coercive:rob:lossyScalar}, the fields $\varepsilon_c$, $\mu$ and $\alpha$ are constant, then, for all coercivity direction $\theta_\mathrm{coe}\in\Theta_\mathrm{coe}$, the coercivity constant of the sesquilinear form $a^+$ can be chosen equal to
    \begin{align}
        C_\mathrm{coe}^+ =
        \min\left( - \mu^{-1} \cos(\theta_\mathrm{coe}), \
        \omega^2 |\varepsilon_c| \cos(\theta_\mathrm{coe}+\varphi_{\varepsilon_c}), \
        |\alpha| \cos(\theta_\mathrm{coe}+\varphi_\alpha)
        \right).
    \end{align}
\end{corollary}
\begin{proof}
    For all $\theta_\mathrm{coe}\in\Theta_\mathrm{coe}$, the sesquilinear form \eqref{eqn:isotrop:robin:sesqform} gives successively
    \begin{align}
        \Re\{a^+(-e^{\im\theta_\mathrm{coe}}\bu,\bu)\}
         & = - \Re\{e^{\im\theta_\mathrm{coe}} (\mu^{-1} \curl\bu,\curl\bu)_\Omega\}
        + \omega^2 \Re\{e^{\im\theta_\mathrm{coe}} (\varepsilon_c \bu,\bu)_\Omega\}                                           \\
         & \qquad\qquad
        + \Re\{e^{\im\theta_\mathrm{coe}} (\alpha \gamma^T\bu,\gamma^T\bu)_\Gamma\}                                           \\
         & = - \Re\{\mu^{-1} e^{\im\theta_\mathrm{coe}}\} \left\|\curl\bu\right\|_{\BL^2(\Omega)}^2
        + \omega^2 \Re\{|\varepsilon_c| e^{\im\varphi_{\varepsilon_c}} e^{\im\theta_\mathrm{coe}}\} \|\bu\|_{\BL^2(\Omega)}^2 \\
         & \qquad\qquad
        + \Re\{|\alpha| e^{\im\varphi_\alpha} e^{\im\theta_\mathrm{coe}}\} \left\|\gamma^T\bu\right\|_{\BL^2(\Gamma)}^2       \\
         & =
        - \mu^{-1} \cos(\theta_\mathrm{coe}) \left\|\curl\bu\right\|_{\BL^2(\Omega)}^2
        + \omega^2 |\varepsilon_c| \cos(\theta_\mathrm{coe}+\varphi_{\varepsilon_c}) \left\|\bu\right\|_{\BL^2(\Omega)}^2     \\
         & \qquad\qquad
        + |\alpha| \cos(\theta_\mathrm{coe}+\varphi_\alpha) \left\|\gamma^T\bu\right\|_{\BL^2(\Gamma)}^2                      \\
         & \geq C_\mathrm{coe}^+ \left\|\bu\right\|_{\BH^+(\curl,\Omega)}^2,
    \end{align}
    which gives the result. By direct inspection, $C_\mathrm{coe}^+>0$.
\end{proof}


\subsection{Geometric media}

Anisotropic metamaterials can be obtained by coordinate transformations in particular reference frames.
Cloaks of invisibility and perfectly matched layers (PMLs) are typical examples.

Let us consider a spherical layer centered at the origin.
The interior and exterior radius of the layer are $R_1$ and $R_2$, respectively, with $0<R_1<R_2$.
The material tensors obtained by using a transformation in spherical coordinates can be written as $\EPS = \varepsilon_0 \tens{\Lambda}$ and $\MU = \mu_0 \tens{\Lambda}$ with
\begin{align}
    \tens{\Lambda} & = \frac{s_2^2}{s_1} (\hat{\br}\otimes\hat{\br}) + s_1 (\tens{I}_3 - \hat{\br}\otimes\hat{\br}),
\end{align}
where $r\in{]R_1,R_2[}$ is the radial coordinate, $\hat{\br}$ is the radial unit vector, and $s_1(r)$ and $s_2(r)$ are spatially-varying coefficients.
This tensor field is normal.
For each $r\in{]R_1,R_2[}$, the eigenvalues are $s_1$ (with multiplicity $2$) and $s_2^2/s_1$.

The goal of a cloak of invisibility is to hide an object from any incident waves.
The object is placed inside the cloak (i.e.~inside the sphere of radius $R_1$), and all incident waves must be perfectly transmitted through the cloak without seeing the object.
For a spherical cloak, the coefficients are
\begin{align}
    s_1(r) = 1/(1-R_1/R_2)
    \qquad\text{and}\qquad
    s_2(r) = (1-R_1/r)/(1-R_1/R_2),
\end{align}
see e.g.~\cite{pendry2006controlling} and \cite[chapter 8]{cai2010optical}.
The eigenvalues are strictly positive, and then the material tensors are symmetric positive definite, inside the layer (i.e.~for $r\in{]R_1,R_2[}$).
However, it is not possible to find a strictly positive lower bound for $s_2$, because $s_2$ is smooth and equal to $0$ at $r=R_1$.
If we consider a slightly modified model where $s_2(r)$ is replaced by $\widetilde{s}_2(r) = \max(s_2(r),\epsilon)$ with $0<\epsilon\ll1$, then the material tensors are elliptic.

The PMLs are artificial absorbing layers that are commonly used to truncate computational domains.
They can be obtained by using coordinate transformations with an imaginary part in the complex plane, see e.g.~\cite{chew1997complex,vaziri2019perfectly}.
For a spherical PML located within the layer $]R_1,R_2[$ to truncate the spherical domain of radius $R_1$, the coefficients $s_1$ and $s_2$ are
\begin{align}
    s_1(r) = 1 + \im\sigma(r)/\kappa
    \qquad\text{and}\qquad
    s_2(r) = 1 + \im/(\kappa r) \int_{R_1}^r \sigma(r')\:dr',
\end{align}
where $\sigma(r)$ is a positive real function called the \emph{absorbing function}, and $\kappa=\sqrt{\varepsilon_0 \mu_0}\:\omega$.
We can write $s_1=|s_1|\:e^{\im\varphi_1}$ and $s_2=|s_2|\:e^{\im\varphi_2}$ with $\varphi_1,\varphi_2\in{[0,\pi/2[}$, because the real and imaginary parts of the coefficients are, respectively, strictly positive and positive.
Then, the eigenvalues of $\tens{\Lambda}$ are $s_1=|s_1|\:e^{\im\varphi_1}$ and $s_2^2/s_1 = |s_2|^2/|s_1|\:e^{\im(2\varphi_2-\varphi_1)}$.
Note that the angle between the phases of these eigenvalues, i.e.~$2|\varphi_2-\varphi_1|$, is always strictly lower than $\pi$.
Therefore, if the absorbing function is bounded, all the eigenvalues belong to an open half space of $\mathbb{C}$ that does not contain zero, and then the tensors are elliptic (see Proposition \ref{prop:ellipticityTensNorm}).
On the other hand, unbounded absorbing functions are often used to improve the accuracy of domain truncation, e.g.~\cite{bermudez2008exact}.
Problems with unbounded media do not fit into the framework proposed in this work, but some analysis could be done using weighted functional spaces.


\subsection{Gyrotropic media}

A medium is \emph{gyroelectric} or \emph{gyromagnetic} if there is a rectangular reference frame where $\EPS$ or $\MU$, respectively, can be represented as
\begin{align}
    \EPS = \varepsilon_0
    \begin{bmatrix}
        \varepsilon_1     & \im\varepsilon_2 & 0             \\
        -\im\varepsilon_2 & \varepsilon_1    & 0             \\
        0                 & 0                & \varepsilon_3 \\
    \end{bmatrix}
    \qquad\text{or}\qquad
    \MU = \mu_0
    \begin{bmatrix}
        \mu_1     & \im\mu_2 & 0     \\
        -\im\mu_2 & \mu_1    & 0     \\
        0         & 0        & \mu_3 \\
    \end{bmatrix},
\end{align}
see e.g.~\cite[section 6.6]{someda2017electromagnetic} and \cite{geng2011extended}.
The eigenvalues of these matrices are $\{\varepsilon_0(\varepsilon_1\pm\varepsilon_2),\varepsilon_0\varepsilon_3\}$ and $\{\mu_0(\mu_1\pm\mu_2),\mu_0\mu_3\}$.

Magnetized ferrite is an example of gyromagnetic medium.
We have $\EPS = \varepsilon \tens{I}_3$ and
\begin{align}
    \mu_1 = 1 - (\omega_0\omega_M)/(\omega^2-\omega_0^2),
    \qquad
    \mu_2 = (\omega\omega_M)/(\omega^2-\omega_0^2)
    \qquad\text{and}\qquad
    \mu_3 = 1,
\end{align}
where $\omega_0$ is the gyromagnetic response frequency and $\omega_M$ is the saturation magnetization frequency, see e.g.~\cite[section 8.18]{ishimaru2017electromagnetic} and \cite[section 6.9]{someda2017electromagnetic}.
The eigenvalues of $\MU$ are $\{\mu_0+\mu_0\omega_M/(\omega_0\pm\omega),\mu_0\}$.
The matrix $\MU$ is Hermitian positive definite if and only if $1+\omega_M/(\omega_0\pm\omega)>0$.

Cold plasma with a static magnetic field ensuring plasma stability is an example of gyroelectric medium, see e.g.~\cite[section 6.2]{bellan2008fundamentals}, \cite[section 6.8]{someda2017electromagnetic} and \cite{stix1992waves}.
S\'ebelin et al~\cite{sebelin1997uniqueness} studied the following model.
The magnetic permeability tensor is $\MU = \mu_0 \tens{I}_3$ and the coefficients of the dielectric tensor $\EPS$ are
\begin{align}
    \varepsilon_1 = 1 + \sum_\varsigma  \frac{\alpha}{\omega} \frac{\omega_{p\varsigma}^2}{\omega_{c\varsigma}^2-\alpha^2},
    \qquad
    \varepsilon_2 = \sum_\varsigma \frac{\omega_{c\varsigma}}{\omega} \frac{\omega_{p\varsigma}^2}{\omega_{c\varsigma}^2-\alpha^2}
    \qquad\text{and}\qquad
    \varepsilon_3 = 1 - \sum_\varsigma \frac{\omega_{p\varsigma}^2}{\omega\alpha},
\end{align}
with $\alpha=\omega+\im\nu$, where $\omega_{p\varsigma}$ and $\omega_{c\varsigma}$ are the plasma frequency and the cyclotron frequency of each species $\varsigma$ (ions and electrons), and $\nu>0$ is the electron-ion collision frequency. 
The frequency $\omega_{c\varsigma}$ is negative for electrons.
S\'ebelin et al~\cite{sebelin1997uniqueness} studied the well-posedness of a variational formulation with a coercive sesquilinear form for the Dirichlet problem.
The well-posedness of several alternative formulations was studied by Back et al \cite{back2015electromagnetic} for a model with an additional Landau damping effect.
For this model, the problem can be studied with our approach, see \cite{chicaud2021analysis,chicaud2021phd}.

\section{Conclusion}

In this work, we have reviewed, extended and applied recent results for the analysis of time-harmonic electromagnetic boundary value problems in elliptic media.
This family of media includes heterogeneous, complex and anisotropic media, which are of practical interest as illustrated by the examples of Section~\ref{sec:applications}.

The considered mathematical framework relies on the use of a general definition of elliptic fields (Definition \ref{def:ellipticityTens}).
We have systematically studied properties of general, normal and hermitian tensor fields, as well as scalar fields, that verify the ellipticity condition.
In particular, we have identified ellipticity directions and, in some cases, we have characterized the sets of all the possible ellipticity directions.

General well-posedness and regularity results have been obtained in \cite{chicaud2021analysis,chicaud2023analysis} by using this framework.
Here, we have more specifically studied conditions such that the problems are coercive.
The coercivity directions of the sesquilinear forms are related to the ellipticity direction of the material fields, and the coercivity constants depend on the constants used in the ellipticity condition of the material fields.
These results could be used to optimize the coercivity conditions, which are related to the stability of the formulations.

While we have mainly studied the electromagnetic problems at the continuous level, our results could be used to analyze numerical methods and domain decomposition methods for solving these problems.
Indeed, these analyses generally depend on the properties of the initial differential problems.
In the coercive case, the coercivity constant of the sesquilinear form plays an important role in the stability of the solution methods.

\section*{Acknowledgement}

This work was supported in part by the \emph{ANR JCJC project WavesDG} (research grant ANR-21-CE46-0010) and by the \emph{Agence de l'Innovation de Défense} [AID] through \emph{Centre Interdisciplinaire d'Etudes pour la D\'efense et la S\'ecurit\'e} [CIEDS] (project 2022 ElectroMath).

\small
\linespread{0.25}
\setlength{\bibsep}{1pt plus 0ex}
\bibliographystyle{abbrvnat}
\bibliography{myrefs}
\addcontentsline{toc}{section}{References}

\end{document}